\documentclass[12pt,letter paper]{amsart}

\usepackage{cite}
\usepackage{amssymb,amsthm,amsmath,mathrsfs}
\usepackage[utf8]{inputenc}
\usepackage[english]{babel}
\usepackage{graphicx,enumerate, microtype, fullpage, comment}
\usepackage{xcolor, hyperref, color, tikz, mathtools}
\usepackage{cleveref, tikz-3dplot}
\usepackage[hang,flushmargin]{footmisc} 
\usepackage[normalem]{ulem}

\usetikzlibrary{3d, backgrounds, fit, cd}
\usetikzlibrary{positioning, matrix, arrows, decorations.pathmorphing}

\tikzstyle{invisivertex} = [black, shape=rectangle, minimum size=0pt, inner sep=3pt]
\tikzstyle{circledvertex} = [violet!80, draw, shape=circle, minimum size=0pt, inner sep=3pt]
\tikzstyle{dot} = [black, fill, shape=circle, minimum size=3pt, inner sep=2pt]
\tikzstyle{opendot} = [draw, black, shape=circle, minimum size=6pt, inner sep=1pt]
\tikzstyle{edgelabel} = [magenta, shape=circle, minimum size=6pt, inner sep=1pt]

\newcommand{\cL}{\mathcal{L}}

\newcommand{\cN}{\mathcal{N}}

\newcommand{\R}{\mathbb{R}}

\newcommand{\rank}{\mathrm{rank}}
\newcommand{\spn}{\mathrm{span}}

\newcommand{\Dfn}[1]{\emph{\bfseries #1}}

\newtheorem{theorem}{Theorem}
\newtheorem{corollary}[theorem]{Corollary}
\newtheorem{claim}[theorem]{Claim}
\newtheorem{proposition}[theorem]{Proposition}
\newtheorem{lemma}[theorem]{Lemma}

\theoremstyle{definition}
\newtheorem{definition}[theorem]{Definition}

\theoremstyle{remark}
\newtheorem{example}[theorem]{Example}
\newtheorem{note}[theorem]{Note}
\newtheorem{remark}[theorem]{Remark}
\newtheorem{observation}[theorem]{Observation}

\crefname{observation}{Observation}{Observations}
\crefname{claim}{Claim}{Claims}

\title[Line Shellings of Geometric Lattices]{Line Shellings of Geometric Lattices}

\author[Backman]{Spencer Backman}
\address{University of Vermont, Department of Mathematics and Statistics}
\email{\href{mailto:spencer.backman@uvm.edu}{spencer.backman@uvm.edu}}

\author[Dorpalen-Barry]{Galen Dorpalen-Barry}
\address{Texas A\&M University, Department of Mathematics}
\email{\href{mailto:dorpalen-barry@tamu.edu}{dorpalen-barry@tamu.edu}}

\author[Nathanson]{Anastasia Nathanson}
\address{University of Minnesota, School of Mathematics}
\email{\href{mailto:natha129@umn.edu}{natha129@umn.edu}}

\author[Partida]{Ethan Partida}
\address{Brown University, Department of Mathematics}
\email{\href{mailto:ethan_partida@brown.edu}{ethan\_partida@brown.edu}}

\author[Prime]{Noah Prime}
\address{University of Vermont, Department of Mathematics and Statistics}
\email{\href{mailto:noah.prime@uvm.edu}{noah.prime@uvm.edu}}

\begin{document}

\begin{abstract}
Inspired by Bruggesser--Mani's line shellings of polytopes, we introduce line shellings for the lattice of flats of a matroid: given a normal complex for a Bergman fan of a matroid induced by a building set, we show that the lexicographic order of the coordinates of its vertices is a shelling order.  This gives a new proof of Bj\"orner's classical result that the order complex of the lattice of flats of a matroid is shellable, and demonstrates shellability for all nested set complexes for matroids.
\end{abstract}

\maketitle

\section{Introduction}

We establish a link between two landmark results in the theory of shellability.  In 1971, Bruggesser--Mani proved that the boundary complex of a polytope is shellable \cite{bruggesser-mani}.  In 1980, Bj\"orner proved that the order complex of the lattice of flats of a matroid is shellable \cite{bjorner80}.    In the present article, we take direct inspiration from Bruggesser--Mani’s line shellings of polytopes for producing a new shelling of the order complex of a lattice of flats.     Our order applies to any nested set complex on the lattice of flats of a matroid, thus establishing their shellability for the first time.  

Our work takes place in the setting of tropical geometry.  Sturmfels showed that the tropicalization of a linear space depends only on the underlying matroid \cite{sturmfels2002solving}.  Ardila--Klivans introduced the Bergman fan of a matroid as a generalization of the tropicalization of a linear space \cite{ArdilaKlivans}, and they showed that the Bergman fan is triangulated by the order complex of the lattice of flats, thus providing a robust geometric realization for this classical abstract simplicial complex.

In their seminal work introducing wonderful compactifications of hyperplane arrangement complements, De Concini--Procesi defined building sets and nested set complexes on the intersection lattice of a hyperplane arrangement \cite{deConcini-Procesi}.  Feichtner--Kozlov extended these notions to more general posets \cite{FK}, and Feichtner--Sturmfels strengthened the result of Ardila--Klivans by demonstrating that the Bergman fan of a matroid is triangulated by any nested set complex on the lattice of flats \cite{feichtner-sturmfels}\footnote{The order complex is the nested set complex associated to the maximal building set.}.

The Bergman fan is central to Adiprasito--Huh--Katz's celebrated proof of the Heron--Rota--Welsh conjecture on the log-concavity of the coefficients of the characteristic polynomial of a matroid \cite{AHK}.  Those authors demonstrate that the Chow ring of a matroid, as introduced by Feichtner--Yuzvinsky \cite{FY}, behaves very much like the cohomology ring of a smooth projective toric variety, despite the fact that the Bergman fan is not complete.  Their work is built around the combinatorial ample cone of a matroid, which is the space of convex functions on the Bergman fan.  Such functions naturally generalize the space of polytopes whose normal fan is a fixed projective fan, although these ample classes no longer admit a natural interpretation via convex bodies.  

Later, Nathanson--Ross demonstrated that for a distinguished nonempty subcone of the ample cone, these convex functions admit a geometric interpretation as 
objects they call  \emph{cubical normal complexes}\footnote{For brevity's sake, we may refer to cubical normal complexes simply as normal complexes.}, and they showed that the degree of these classes in the Chow ring agrees with the classical volume of the corresponding normal complexes \cite{NR}.  In subsequent work, Nowak--O'Melveny--Ross utilized normal complexes for giving a ``volume proof'' of the Heron--Rota--Welsh conjecture \cite{nowak2025mixed}.

We further develop this close analogy between normal complexes for Bergman fans and polytopes by demonstrating that normal complexes can be utilized for producing lexicographic line shellings of Bergman fans.  

\begin{theorem}\label{thm1} Let $M$ be a matroid, $B$ a building set for the lattice of flats of $M$, and $N$ the associated nested set complex.  Let $\Sigma$ be the Bergman fan triangulated by $N$, and take $P$ a normal complex for $\Sigma$.  The lexicographic order on the coordinates of the vertices of $P$ is a shelling order for $N$.   
\end{theorem}

While line shellings of polytopes are determined by generic linear functionals, it is notable that Theorem \ref{thm1} specifically utilizes a lexicographic order \footnote{It is well-known that lexicographic orders are determined by certain (generic) linear functionals.}.  As demonstrated in Section \ref{sec:examples}, genericity is not sufficient for producing line shellings of Bergman fans via normal complexes.  Our choice to pursue this lexicographic order is motivated by Bj\"orner's EL-labeling and the tropical nature of our construction.  Although our order is different from Bj\"orner's order in the case of the maximal building set, the two are locally similar and we leverage a generalization of this connection (see Proposition \ref{claim:dimension-1-minimal}) in our proof of Theorem \ref{thm1}.

The proof of Theorem \ref{thm1} proceeds by induction on the dimension of the normal complex, blending together geometry and combinatorics.  The key fact which we employ is that each \emph{facet} of a normal complex is itself a normal complex for a smaller matroid,
thus our argument mimics the inductive verification of line shellings for polytopes (see the proof of Theorem \ref{lineshellingtheorem}).   In Remark \ref{recursiveELremark}
we explain how Bj\"orner's EL-shelling admits a recursive proof which draws these perspectives together.

For extending our result beyond the setting of order complexes to general nested set complexes for matroids, we introduce a natural combinatorial generalization of Bj\"orner's EL-shelling order of the order complex of the lattice of flats, which is defined for an arbitrary nested set complex.
We call this order the nested lexicographic order (NL-order).  This aspect of our proof, utilizing the NL-order, is somewhat idiosyncratic. 
At the time of writing, we do not know whether the NL-order itself is a shelling order.  We are able to sidestep this uncertainty; our proof only requires a weaker property of the NL-order which we are able to prove directly.
In the future, we hope to resolve the question of whether the NL-order is indeed a shelling order.

Prior to this work, it was known that each nested set complex is Cohen--Macaulay, a weaker property than being shellable.  This follows from Munkres' theorem that the CM-property is a topological invariant \cite{munkres1984topological}, Bj\"orner's theorem that the order complex of the lattice of flats is shellable, and Feichtner--M\"uller's theorem that all nested set complexes are obtainable from the order complex by a sequence of combinatorial blow-downs \cite{FM}.

There are several different nested set complexes for matroids which appear in the literature, and some have previously been shown to be shellable.  Trappmann--Ziegler showed that the poset of strata of $\overline{M}_{0,n}$ is shellable \cite{trappmann-ziegler, feichtner} \footnote{This the nested set complex associated to the minimum building set for the braid matroid.}.  Braden--Matherne--Proudfoot--Huh--Wang introduced the augmented Bergman fan of a matroid \cite{braden-huh-matherne-proudfoot-wang}, and Bullock--Kelley--Reiner--Ren--Shemy--Shen--Sun--Tao--Zhang showed that the augmented Bergman fan is shellable  \cite{minnesota-REU-2021}.  Crowley--Huh--Larson--Simpson--Wang introduced the polymatroid Bergman fan \cite{crowley-huh-larson-simpson-wang}.  Our main Theorem \ref{thm1} gives a uniform proof of shellability for all of these complexes.

While this work was in the final stages of preparation, we learned of the independent work of Coron--Ferroni--Li who establish \emph{vertex decomposability} of nested set complexes for matroids, which also implies their shellability \cite{coron-ferroni-li}.  To the best of our knowledge their constructions do not appear to be immediately related to the tropical polyhedral perspective presented here.

The remainder of this paper is organized as follows.
In Section \ref{sec:background}, we present the relevant background on shellings, building sets, nested set complexes, Bergman fans, and normal complexes.
In Section \ref{sec:examples}, we illustrate the main result and some of its subtleties with several examples.  
In preparation for the proof of Theorem \ref{thm1}, we present several combinatorial results in Section \ref{sec:links}.
The proof of Theorem \ref{thm1} is contained in Section \ref{sec:proof-of-main}.

\subsection{Additional Related Works}

We describe here some further connections to works in the literature.   While this work was in preparation, Balla--Joswig--Weis posted a preprint where they utilize line shellings for proving that tropical hypersurfaces are shellable \cite{balla2025shellings}-- their line shellings do not make use of normal complexes.  To the best of our knowledge, these two works are the only examples of line shellings being utilized for proving shellability of non-complete fans.

In work of Amini and Piquerez, ``Homology of Tropical Fans'' \cite{amini2021homology}, those authors introduce a notion of a \emph{shellable tropical fan}.  It is unclear at the time of writing what is the precise relationship between their fans and the classical notion of shellability.  In particular, it was observed by June Huh \footnote{This observation was made during the BIRS Workshop ``Algebraic Aspects of Matroid Theory'' March 11-17, 2023.} that Amini and Piquerez's shellable tropical fans include all complete fans, and it is a major open question in polyhedral geometry whether such fans are always shellable in the traditional sense.

In work of Adiprasito-Bj\"orner \cite{AB}, the authors investigate combinatorial tropical Lefschetz section theorems.  An important object in that work is the positive side of a Bergman fan.  Confirming a conjecture of Mikhalkin-Ziegler, they  prove that this fan is Cohen-Macaulay and ask whether it is always shellable.  It would be interesting to see if the techniques we introduce in this article can be applied to make progress on their question.

In work of Heaton and Samper \cite{heaton2020dual}, line shellings of (dual) matroid polytopes were investigated. They explore connections between such line shellings and the theory of matroid activities.  Beyond their work and ours, we are not aware of other instances where line shellings of polyhedral complexes have been investigated in the setting of matroid theory.

\subsection*{Acknowledgments} We thank Federico Ardila-Mantilla, Eva-Marie Feichtner, Carly Klivans, June Huh, Vincent Pilaud, Dasha Poliakova, Vic Reiner, Dustin Ross, Raman Sanyal, and David Speyer for helpful conversations and comments.  We thank Basile Coron, Luis Ferroni, and Shiyue Li for informing us of their work and for coordinating the posting of our articles.  The first author was supported by NSF Grant (DMS-2246967) and Simons Foundation Gift \# 854037.  The third author received partial support from NSF Grant (DMS-2054436).  The fourth author was partially supported by NSF Grant (DMS-2053288), a U.S. Department of Education GAANN award, and the Simons Foundation (SFI-MPS-SDF-00015018).

\section{Background}\label{sec:background}
In this section, we review some constructions from topological combinatorics, matroid theory, tropical geometry, and polyhedral geometry.

\subsection{Shellings}
We recall shellings of polytopal complexes.
For more background on shellings of polytopal complexes, we refer the reader to \cite[Chapter 8]{ziegler}. For more information on shellings of order complexes of posets, we refer the reader to Wachs \cite{wachs}.

\begin{definition}[Shellable Polytopal Complex](See  \cite[Definition 8.1]{ziegler}) 
Let $P$ be a pure polytopal complex.
A linear order $\prec$ on the facets of $\Delta$ is a \Dfn{shelling order} if either $P$ is zero-dimensional (meaning that all facets are points) or it satisfies the following two conditions
    \begin{enumerate}[(i)]
        \item The boundary complex of the first facet admits a shelling order, and 
        \item For each facet $F$, the boundary complex of the intersection of $F$ with all previous facets under $\prec$ is a shellable polytopal complex of dimension one less, i.e., the boundary of 
        \[
        F \cap \bigcup_{\substack{F': facet \\ F' \prec F}} F'
        \]
        is dimension one less than $P$ and admits a shelling order.
    \end{enumerate}
A pure polytopal complex $P$ is \Dfn{shellable} if it admits a shelling order \footnote{Shellability for non pure complexes was introduced by Bj\"orner--Wachs \cite{BjornerWachs-shelling}.}.  
\end{definition}

Let $P$ be a pure polytopal complex.  Let $V$ be a Euclidean vector space and $\{v_N \in V : N \text{ facet of } \Delta\}$ a collection of distinct points in $V$.
If there exists a vector $\gamma$ such that the total order
\[
N \prec N' \iff \langle v_N, \gamma \rangle < \langle v_{N'}, \gamma \rangle 
\]
is a shelling order for $\Delta$, then we call this order a \Dfn{line shelling} of $\Delta$.
Bruggesser and Mani {\cite[Proposition 2]{bruggesser-mani}} demonstrated that the boundary complex of a polytope admits a line shelling by taking the inner product of the vertices of the \emph{dual polytope} with a fixed \emph{generic vector}; see also \cite[Sections 3.1,3.4,8.2]{ziegler}.

For the purposes of this article, it is conceptually useful to package line shellings of polytopes in the following way.

\begin{theorem}[\!\!{{\cite[Proposition 2]{bruggesser-mani}}}]\label{lineshellingtheorem}\
Let $P \subsetneq \mathbb{R}^n$ be a polytope and $\Sigma$ its normal fan.  Let $\gamma \in \mathbb{R}^n$ be a generic vector.  The order of the vertices of $P$ according to their inner product with $\gamma$ is a line shelling of $\Sigma$.
\end{theorem}

When a polytopal complex is simplicial, there is a simpler description of a shelling order (see \cite[Equation 7.3]{bjorner}).

\begin{definition}[Shellable Simplicial Complex]
Let $\Delta$ be a pure simplicial complex.
A linear order $\prec$ on the facets of $\Delta$ is a \Dfn{shelling order} if for all facets $N,N'$ with $N \prec N$', there exists some $N''$ such that $N'' \prec N$ and $X \in N$ such that
\[
N \cap N' \subseteq N \cap N'' = N \setminus\{X\}\,.
\]
A simplicial complex $\Delta$ is \Dfn{shellable} if it admits a shelling order.    
\end{definition}

We briefly sketch a proof of Theorem \ref{lineshellingtheorem} in the case when $P$ is simple (so that $\Sigma$ is simplicial).  This proof will be a model for our own proof of Theorem \ref{thm1}.

\begin{proof}
We proceed by induction on $P$ with the base case being a point.  Let $N$ and $N'$ be chambers of the normal fan with $N \prec N'$.  We wish to find some chamber $N''$ with $N'' \prec N'$ such that $N\cap N' \subseteq N \cap N'' = N \setminus \{\rho\}$ for some ray $\rho \in N'$.  Case 1: $N \cap  N' =\emptyset$.  Here we use the fact that $\gamma$ induces an acyclic orientation of the 1-skeleton of $P$.  Let $v$ be the vertex corresponding to $N'$.  As $v$ is not the first vertex in our order, there is some edge $\overline{wv}$ which is oriented towards $v$.  We can take $N''$ to be the chamber dual to the vertex $w$.  Case 2: There exists some ray ${\Tilde{\rho}} \in  N \cap  N'$.  We may look at the facet of $P$ which is normal to ${\Tilde{\rho}}$ (this facet contains the vertices corresponding to $N$ and $N'$), and apply induction on dimension.
\end{proof}

The goal of this article is to show how the above theorem extends to a certain collection of non-complete fans, namely Bergman fans of matroids equipped with building sets. The following classical order will be of essential interest for us.

\begin{definition}\label{def:usual-lex-order}
    Let $E$ be a set and $<_E$ be any linear order on $E$.
    The \Dfn{lexicographic order} on $k$-tuples of elements in $E$ is a linear order defined by $(a_1,\dots, a_k) <_{\text{lex}} (b_1,\dots, b_k)$ if and only if there is some $m \geq 0$ such that $a_m <_E b_m$ and $a_i  = b_i$ for $i < m$.
\end{definition}

The following definition is due to Bj\"orner \cite[Section 2]{bjorner80}.

\begin{definition}[EL-labeling of a Poset]\label{def:el}
Let $\cL$ be a bounded poset. 
We use $\mathscr{E}$ to denote the edges of the Hasse diagram of $\cL$.
A map $\lambda: \mathscr{E} \to \mathbb{Z}_{\geq 0}$ is an \Dfn{edge-lexicographic labeling} if for every closed interval $[X,Y]$ of $\cL$ there is a unique maximal chain whose label sequence is strictly increasing and this chain lexicographically precedes all other maximal chains in the closed interval.
\end{definition}

If such an EL-labeling exists, then Björner proves that the lexicographic order of the EL-labelings of the maximal chains is a shelling order for the order complex of $\cL$ (we may denote this order by $\prec_{EL}$) \cite[Theorem 2.3]{bjorner80}.
When $\cL$ is a geometric lattice, Björner gives a simple construction to produce an EL-labeling, thus establishing their shellability.
Let $\cL_1$ be the set of atoms of $\cL$ (we will suppose for simplicity that $\cL_1$ is labeled by $\mathbb{Z}_{>0}$).
Define a labeling $\lambda: \mathscr{E} \to \mathbb{Z}_{>0}$ sending the edge $(X,Y)$ to the smallest atom below $Y$ but not below $X$.
This is an EL-labeling of the geometric lattice \cite[Theorem 3.7]{bjorner80}.

In the process of proving Theorem \ref{thm1}, we will introduce and relate two different facet orders, one combinatorial (the nested lexicographic order) and one geometric (the normal complex order).
Motivated by our considerations we introduce the following natural definitions -- we will eventually demonstrate our two orders are \emph{weakly locally equivalent}.

\begin{definition}[Local Equivalence]\label{def:locally-equivalent}
Let $\prec_A$ and $\prec_B$ be two linear orders on the facets of a simplicial complex $\Delta$.
If, for every codimension $1$ face $F$ of $\Delta$, the restrictions of $\prec_A$ and $\prec_B$ to the facets of $\Delta$ containing $F$ coincide, then we say that these two orders are \Dfn{locally equivalent}.
\end{definition}

\begin{definition}[Weak Local Equivalence]\label{def:weakly-locally-equivalent}
Let $\prec_A$ and $\prec_B$ be two linear orders on the facets of a simplicial complex $\Delta$.
If, for every codimension one face $F$ of $\Delta$, the restrictions of $\prec_A$ and $\prec_B$ to the facets of $\Delta$ containing $F$ have the same minimum element then we say that the two orders are \Dfn{weakly locally equivalent}.
\end{definition}

\subsection{Matroids and the nested set complex}

We refer the reader to Oxley \cite{oxley} for an introduction to matroid theory.  Let $M$ be a loopless matroid on 
a finite ground set $E$, and let $\cL(M)$ be the lattice of flats of $M$.
When no confusion will arise, we may write $\cL$ instead of $\cL(M)$.
We do not assume our matroids are simple, so we view $\cL(M)$ and $\cL$ as \emph{labeled} posets.  
Given a flat $X$ of $\cL$, the \Dfn{restriction} of $\cL$ to $X$ is defined as the matroid whose flats are $\cL|_X = \{Y\in \cL \colon Y\leq X\}$ and the \Dfn{contraction} of $\cL$ along $X$ is defined as the matroids whose flats are $\cL^X = \{(Y\setminus X) \colon X \leq Y \}$.

A \Dfn{building set} $B \subseteq \cL \setminus \{\emptyset\}$ is a subset of the flats of $\cL$ such that for all $X\in \cL\setminus \{\emptyset\}$, the map
\begin{align*}
\left(\prod_{Y \in \,\max (B_{\leq X})} [\emptyset,Y] \right) & \to [\emptyset,X]\\
(Z \in [\emptyset, Y]  \colon Y \in \max (B_{\leq X})) & \mapsto \bigvee_{Y \in \max (B_{\leq X})} Z\,,
\end{align*}
is a poset isomorphism, where $\max (B_{\leq X})$ denotes the containment-maximal elements of $B$ that lie weakly below\footnote{This definition is often stated with a strict inequality and $X \not\in B$. Here we drop the strictness, but note that if $X\in B$ then $\max(B_{\leq X}) = \{X\}$ and the map is just the identity.} the flat $X$.
We use $\max(B) := \{F_1, \ldots, F_m\}$ to denote the containment-maximal elements of the building set.
There is an alternate characterization of building sets, which we will find useful in our proofs.

\begin{proposition}[\!\!\!{{\cite[Section 2.3]{deConcini-Procesi}}{\cite[Proposition 2.11]{backman-danner}}}]\label{prop:backman-danner}  Let $M$ be a matroid with lattice of flats $\cL$, then 
$B \subseteq \cL \setminus \{\hat{0}\}$ is a building set if and only if $B$ contains the connected flats of $\cL$, and for all $X,Y \in B$ with $X \cap Y \not= \emptyset$, we have $X\vee Y \in B$.
\end{proposition}

We now introduce some of the principal objects of our study. Namely, the nested sets of a building set. Note that we follow Postnikov's convention \cite[Definition 7.3]{postnikov} and require that all nested sets contain $\max(B)$.

\begin{definition}\label{defn:nested}
    A subset $N$ of a building set $B$ is \Dfn{nested} if it satisfies the following conditions:
    \begin{enumerate}
        \item $N$ contains $\max(B)$.
        \item For any collection of $\ell\geq 2$ pairwise incomparable flats $X_1,\ldots X_\ell$ of $N$, the join $\bigvee_{i=1}^\ell X_i$ is not in $B$.
    \end{enumerate}
A subset $N$ of $B\setminus \max(B)$ is \Dfn{reduced nested} if $N \cup \max(B)$ is nested.
\end{definition}

\begin{observation}\label{obs:forest}
    The lattice of flats of $M$ restricted to the elements of a nested set $N$ forms a forest poset. That is, if $X,Y$ and $Z$ are nonempty flats of $N$ with $X< Y$ and $X< Z$, then either $Y\leq Z$ or $Z\leq Y$.
    To see why this is true, suppose that $Y$ and $Z$ are incomparable. As the $X \subseteq Y \cap Z$, it follows from Proposition \ref{prop:backman-danner} that $Y\vee Z\in B$. This contradicts the assumption that $N$ is nested.
\end{observation}

We will routinely use the following corollary.
\begin{corollary}\label{cor:disjoint}
    If $\{X,Z\}$ is a reduced nested set with $X$ and $Z$ incomparable, then $X\cap Z =\emptyset$.
\end{corollary}

The \Dfn{nested set complex} of $B$, denoted $\Delta(\cL,B)$, is the simplicial complex
\[\Delta(\cL,B) \coloneqq \{ N \subseteq B \setminus\max(B): N \text{ is reduced nested}\}\,. \]

Given a flat $X\in \cL$, the \Dfn{restriction} $B\vert_X$ of $B$ to $X$ and \Dfn{contraction} $B^X$ of $B$ at $X$ are
\begin{align*}
B\vert_X & = \{Y\in B: Y\leq X\} \subseteq \cL\vert_X\\
B^X & = \{(Y\vee X)\setminus X: Y\in B, Y\not\leq X\} \subseteq \cL^X.
\end{align*}
For $X \in B \setminus \max(B)$, one has that both $B\vert_X$ and $B^X$ are building sets in their respective lattices\footnote{It's easy to see that $B\vert_X$ is a building set, even when $X$ is not in $B$.}.

\begin{proposition}[\!\!\!{{ \cite[Theorem 4.3]{deConcini-Procesi},  {\cite[Section 3]{zelevinsky},\cite[Propositions 2.8.6-7]{bibby-denham-feichtner},{\cite[Proposition A.8]{wondertopes}}, {\cite[Proposition 2.40]{Mantovani-Pardol-Pilaud}}}}}]\label{prop:rest_contract}
For $X\in B\setminus \max(B)$, both $B\vert_X \subseteq \cL\vert_X$ and $B^X\subseteq \cL^X$ are building sets.
Moreover, $B|_X$ is a building set for any $X\in \cL$.
\end{proposition}

The following lemma and corollary are standard results about building sets and are well-known to experts.
We record them here to make the exposition easier later.

\begin{lemma}[\!\!\!{{\cite[Proposition 2.8.2]{FK}}}]\label{lem:incomp}
    Let $\{X_1,\ldots, X_k\} \cup \max(B)$ be a nested set of $B$ such that $\{X_1,\ldots, X_k\}$ is an antichain of flats. The set $\{X_1,\ldots,X_k\}$ is the set of inclusion-maximal elements in $B\vert_{\bigvee_{i=1}^k X_i}$. There is an isomorphism
    \[\prod_{i=1}^k[\emptyset, X_i] \simeq [\emptyset, \bigvee_{i=1}^k X_i] \]
    given by the map $(Z_1,Z_2,\ldots, Z_k) \mapsto \bigvee_{i=1}^k Z_i$.
  \end{lemma}
  
\begin{corollary}\label{cor:incomp2}
    If $\{X,Z\}\subseteq B$ is nested with $X$ and $Z$ incomparable, then $(X\vee Z) \setminus Z= X$.
\end{corollary}
\begin{proof}
     \Cref{lem:incomp} tells us that an atom $A$ is contained in $X\vee Z$ if either $A\leq X$ or $A\leq Z$. As every flat is equal to the union of atoms it contains, the claim follows.
\end{proof}

\subsection{Bergman fans}
In this section, we recall the geometry of Bergman fans.
We provide some background on fans, but refer the reader to \cite[Section 1.2]{cox-little-schenck} for further details.

\begin{definition}
A \Dfn{fan} $\Sigma$ in $\R^n$ is a nonempty finite set of polyhedral cones, such that
\begin{enumerate}
    \item Every face of a cone $\sigma\in \Sigma$ is also a cone in $\Sigma$.
    \item The intersection of two cones $\sigma,\sigma' \in \Sigma$ is a face of $\sigma$.
\end{enumerate}
\end{definition}
The \Dfn{lineality space} of a cone $\sigma$ is the largest linear subspace contained in $\sigma$.  The definition of a fan $\Sigma$ implies that all cones in $\Sigma$ have the same lineality space $L$ and that $L$ is the unique inclusion-minimal cone of $\Sigma$. \footnote{While some authors, often those with a background in toric geometry, are inclined to quotient out fans by their lineality space, we will not do so in this article.  This choice is in fact necessary for our main Theorem \ref{thm1} to hold.}
A fan $\Sigma$ is \Dfn{pure} if all maximal cones have the same dimension. A fan $\Sigma$ is \Dfn{unimodular} if all of its cones are unimodular. In this article, all of our fans will be pure and unimodular.

\begin{definition} \label{defn:bergman_fan}
Given a matroid on ground set $E$ with lattice of flats $\cL$, let $\R^E$ denote the real vector space with basis  $\{e_i \colon i\in E\}$, equipped with the standard inner product. 
The \Dfn{Bergman fan} with respect to building set $B$, denoted $\Sigma_{\cL, B}$, is a fan in $\R^E$ with lineality space 
\[
L_B := \spn_\R \{e_X: X\in \max(B)\}
\qquad \text{where } e_{X} = \sum_{i\in X} e_i\,,
\]
and cones 
\[\sigma_N := L_B+\mathrm{cone}(e_X \colon X\in N\setminus \max(B)) \qquad \text{for each nested set } N\subseteq B.\]
\end{definition}

The Bergman fan gives a geometric realization of the nested set complex; the face poset of the Bergman fan, with the bottom element removed, is equal to the face poset of the nested set complex.
In particular, the maximal cones of $\Sigma_{\cL, B}$ correspond to the facets of $\Delta(\cL,B)$, i.e. the inclusion-maximal nested sets.

\subsection{Normal Complexes}\label{sec:normal-complexes}

Recall that every polytope defines a \emph{normal fan}, and this fan is \emph{complete} meaning its support is all of $\mathbb{R}^n$.
The Bergman fan is complete if and only if the underlying matroid is Boolean.  Thus, for any other matroid, the Bergman fan has no normal polytopes.
In this section, we recall certain polytopal complexes 
introduced by the third author and Ross called \emph{cubical normal complexes} -- as we will see, a cubical normal complex is an excellent substitute for a normal polytope for the Bergman fan.

A function $\varphi$ from a fan $\Sigma$
to $\R$ 
is  \Dfn{piecewise linear} 
 if for every cone $\sigma\in \Sigma$, the restriction $\varphi\vert_\sigma$ is  linear on $\sigma$.
 Denote the space of piecewise linear functions on a fan $\Sigma$ as $\mathrm{PL}(\Sigma)$.

The ray generators for each cone of a Bergman fan $\Sigma_{\cL, B}$ are linearly independent, hence the fan is simplicial\footnote{Here we take a slightly relaxed definition of simplicial: for our purposes a fan will be simplicial if it becomes simplicial in the traditional sense after quotienting out by the lineality space.}.  Any piecewise linear function $\varphi$ on a fan is determined by its restriction to a set of ray and lineality space generators.  Conversely, if a fan is simplicial, any function on a set of ray and lineality space generators extends to a piecewise linear function.  Thus, for a simplicial fan $\Sigma$, there is a canonical bijection between the set of functions from a fixed set of ray and lineality space generators to $\mathbb{R}$, and $\mathrm{PL}(\Sigma)$.

\begin{definition}\label{defn:normal_complexes}
For a piecewise linear function $\varphi\in \mathrm{PL}(\Sigma_{\cL, B})$ and $X \in B$, we write $\varphi_X := \varphi(e_X)$. This function defines a hyperplane and associated halfspace
\[H_{X, \varphi} = \{v\in\R^E \colon \langle v, e_X \rangle = \varphi_X \} \qquad \text{and} \qquad H^+_{X, \varphi} = \{v\in\R^E \colon \langle v, e_X \rangle \geq \varphi_X\}.\] 
We say that $\varphi$ is \Dfn{cubical} if, for every nested set $N$ of $B$, we have a single point \footnote{This is a single point by virtue of the Bergman fan being simplicial.}
\[v_N := \sigma_{N}^{\circ} \cap \bigcap_{X\in N} H_{X, \varphi} \]
where $\sigma_N^{\circ}$ is the relative interior of $\sigma_N$.  For $\varphi$ a cubical piecewise linear function, each nested set $N$ defines a polytope
\[P_{N, \varphi} := \sigma_{N} \cap \bigcap_{X\in N\setminus \max(B)} H_{X, \varphi}^+ \cap \bigcap_{Y\in \max(B)} H_{Y, \varphi}.\]
The \Dfn{normal complex} determined by $\varphi$ is the collection of these polytopes:
\[\cN_{\cL, B, \varphi} := \{ P_{N, \varphi}:{N\in \Delta(\cL, B)}\}\,.\]
\end{definition}

\begin{remark}[Why call these ``cubical'' functions?] \label{rem:cubical}
    When $\varphi$ is cubical, each $P_{N,\varphi}$ is combinatorially equivalent to an $n$-cube; see \cite[Proposition 3.8]{NR}.  
    So, when $\varphi$ is cubical, the resulting normal complex is an honest \emph{cubical complex}.
    \footnote{We note for the interested reader that normal complexes are not always CAT(0) cubical complexes.
    Take a normal complex for $U_{3,3}$ with respect to the minimal connected building set, then combinatorially this normal complex is 3 squares on the boundary of a cube which meet at a common vertex -- this is a well-known example of a cubical complex which is not CAT(0).  One should expect that normal complexes are CAT(0) if and only if the nested set complex is \emph{flag}.}).
\end{remark}

\begin{proposition}[\!\!{{\cite[Proposition 7.4]{NR}}}]\label{existnorm}
    There exists a cubical piecewise linear function on the Bergman fan of a matroid $M$ with respect to a building set $B$, thus guaranteeing the existence of a normal complex $N_{\cL, B, \varphi}$.
\end{proposition}

\begin{remark}\label{existenceCNCremark}
We sketch an alternate proof of the existence of normal complexes for Bergman fans.  It was proven by Backman--Danner \cite{backman-danner} and Mantovani--Pilaud--Padrol \cite{Mantovani-Pardol-Pilaud} that, given a Bergman fan $\Sigma$ associated to a matroid $M$ and a building set $B$, there exists a projective Bergman fan $\Sigma'$ such that $\Sigma \subseteq \Sigma'$  \footnote{A Bergman fan is projective if and only if the underlying matroid is a Boolean matroid.}.  Any normal complex for $\Sigma'$ restricts to a normal complex for $\Sigma$, thus it suffices to demonstrate the existence of a normal complex for $\Sigma'$.  

Let $P$ be a permutahedron which is determined by an exponential support function: for $\emptyset \neq X \subseteq E$, let $f(X) = \alpha^{|X|}$ for $\alpha \gg 1$,  then $P$ is the polytope cut out by the inequalities $\sum_{i \in X}x_i \geq f(X)$ and $\sum_{i \in E} x_i = f(E)$. Generalizing the case of graph associahedra in the work of Devadoss \cite{devadoss2009realization}, Pilaud demonstrated  \cite[Remark 25]{pilaud2017nestohedra} (see also \cite{padrol2023deformation}) that for each projective Bergman fan $\Sigma'$, there exists a nestohedron $Q$ normal to $\Sigma'$ such that $Q$ is a \emph{removahedron} for $P$, i.e. we can obtain $Q$ by deleting facets of $P$.  To complete the proof, we observe, perhaps for the first time, that such removahedra are in fact normal complexes.  First note that $P$ is a normal complex -- by symmetry, all of its vertices lie in the corresponding dual chamber of the Bergman fan.  Next, we utilize a result of Feichtner-M\"uller \cite{FM}, in the special case of the Boolean matroid, that $\Sigma'$ can be obtained from the braid arrangement by a sequence of toric blow-downs.  Dually, this implies that the process of deleting facets of $P$ to obtain $Q$ can be done one facet at a time so that all of the intermediate polytopes are nestohedra.  Thus, for completing the proof, it suffices to prove that the cubical normal complex property is preserved under the deletion of a single facet (when that deletion respects the normal fans).  Indeed, it is easy to see that the new vertex introduced by the deletion of this facet lies in the union of the chambers of the normal fan which were merged in the corresponding toric blow down.
\end{remark}

The following re-characterization will prove to be useful to us later.

\begin{lemma}\label{lem:c-is-enough}
    The normal complex $\cN_{\cL, B, \varphi}$ is determined by a single vector $c\in \R^B$.
\end{lemma}

\begin{proof}
Recall that for a simplicial fan, piecewise linear functions are determined by their values on the ray and lineality space generators.
Each of these standard generators has the form $e_X$ for $X \in B$, so $\varphi$ is uniquely determined by a vector in $\R^{B}$.
\end{proof}

\begin{note}\label{note:hxphi-hxc}
In Definition \ref{defn:normal_complexes}, we define hyperplanes $H_{X,\varphi}$ and halfspaces $H_{X,\varphi}^+$.
Each of these depends only on $X$ and the entry $c_X$ of $c$ in the sense of Lemma \ref{lem:c-is-enough}.
We will sometimes denote $H_{X,\varphi}$ and $H_{X,\varphi}^+$ by $H_{X,c}$ and $H_{X,c}^+$ when we want to emphasize that they depend only on the vector $c$.
\end{note}

Normal complexes are polyhedral complexes and thus have faces.  In this article we will utilize a different notion of faces of normal complexes which is motivated by their analogy with polytopes. 

\begin{definition}[Faces and Facets]\label{def:facets}
   Let $\cN := \cN_{\cL,B,\varphi}$ be a normal complex.
   A \Dfn{face} of $\cN$ is a nonempty intersection of $\cN$ with a collection of the hyperplanes $H_{X,\varphi}$.
   A \Dfn{facet} of $\cN$ is an intersection of $\cN$ with a single hyperplane of the form $H_{X,\varphi}$.
\end{definition}

Faces of a normal complex are not polytopes.  However, we will show in Lemma \ref{lem:intersection} that they are themselves normal complexes.

We are now ready to describe the order on the facets of the nested set complexes used in \Cref{thm1}. This ordering is inspired by the following perspective on lexicographic orders.
One way to specify a lexicographic order $<_E$ is to assign (distinct) weights to the elements of $E$.\footnote{This is a common technique in commutative algebra, where term orders in polynomial rings are often described by taking inner products of the exponent vector with a weight vector; see \cite[p.4]{sturmfels}.}
Then $a <_E b$ if and only if the weight of $a$ is less than the weight of $b$.
In that sense, we can think of the lex order induced by a weight vector on the ground set.
This interpretation of lexicographic orders inspires the following definitions.

\begin{definition}\label{def:lexicographic}
    Let $\gamma \in \R^E$ and $\cN$ be a normal complex of $(\cL, B)$. We say that $\gamma$ is \Dfn{lexicographic} on $\cN$ if, for every pair of inclusion-maximal nested sets $N$ and $N'$, $\langle v_N, \gamma \rangle > \langle v_{N'}, \gamma \rangle$ if and only if 
    there exists an index $k$ such that $(v_{N})_k > (v_{N'})_k$ and for all $1\leq i < k$, we have $(v_{N})_i = (v_{N'})_i$.
\end{definition}

\begin{definition}\label{def:normal-complex-order}
A lexicographic vector $\gamma$ on a normal complex $\cN$ gives rise to a total ordering of the inclusion-maximal nested of $(\cL,B)$ by declaring
\[N<_{\cN} N' \qquad \iff \qquad \langle v_N, \gamma \rangle < \langle v_{N'}, \gamma \rangle \]
We call this ordering the \Dfn{normal complex order} (with respect to $\cN$).
\end{definition}

\section{Illustration of Main Result}\label{sec:examples}

Let $M^{\text{br}}$ be the broom matroid, which is the matroid on the ground set $ \{0,1,2,3\}$ with flats
\[
\emptyset, 0, 1, 2, 3, 01, 02, 03, 123, 0123\,.
\]
We illustrate \Cref{thm1} for the minimal building set and the maximal building set of this matroid. We denote these building sets by $B_m$ and $B_M$, respectively. 
The building sets are
\begin{align*}
    B_{m} & = \{0,1,2,3, 123\}\,,\\
    B_{M} & = \{0,1,2,3, 01, 02, 03, 123, 0123\}\,.
\end{align*}
The lattice of flats of $M^{\text{br}}$ is shown below, with the two building sets circled ($B_m$ is on the left and $B_M$ is on the right).
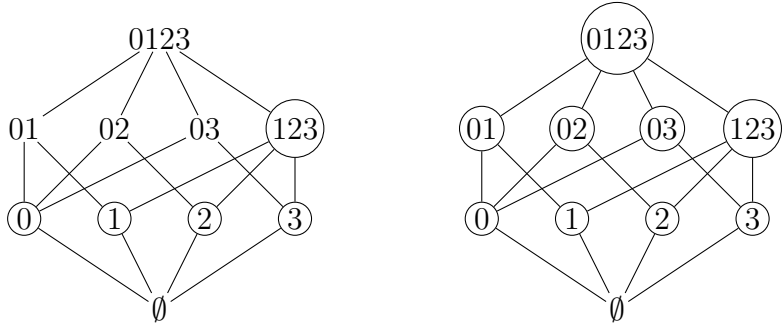
\begin{center}
\begin{tikzpicture}[scale=.6,inner sep=1] 
        \node (em) at (0, 0) {$\emptyset$};
        \node[draw,circle] (0) at (-3,2) {$0$};
        \node[draw,circle] (1) at (-1,2) {$1$};
        \node[draw,circle] (2) at (1, 2) {$2$};
        \node[draw,circle] (3) at (3, 2) {$3$};
        \node (01) at (-3,4) {$01$};
        \node (02) at (-1,4) {$02$};
        \node (03) at (1, 4) {$03$};
        \node[draw,circle] (123) at (3, 4) {$123$};
        \node (e) at (0, 6) {$0123$};
          \foreach \from/\to in {em/0, em/1, em/2, em/3, 0/01, 0/02, 0/03, 1/01, 1/123, 2/02, 2/123, 3/03, 3/123, 01/e, 02/e, 03/e, 123/e}
            \draw (\from) -- (\to);
    \end{tikzpicture}
    \qquad\qquad
    \begin{tikzpicture}[scale=.6,inner sep=1] 
        \node (em) at (0, 0) {$\emptyset$};
        \node[draw,circle] (0) at (-3,2) {$0$};
        \node[draw,circle] (1) at (-1,2) {$1$};
        \node[draw,circle] (2) at (1, 2) {$2$};
        \node[draw,circle] (3) at (3, 2) {$3$};
        \node[draw,circle] (01) at (-3,4) {$01$};
        \node[draw,circle] (02) at (-1,4) {$02$};
        \node[draw,circle] (03) at (1, 4) {$03$};
        \node[draw,circle] (123) at (3, 4) {$123$};
        \node[draw,circle] (e) at (0, 6) {$0123$};
          \foreach \from/\to in {em/0, em/1, em/2, em/3, 0/01, 0/02, 0/03, 1/01, 1/123, 2/02, 2/123, 3/03, 3/123, 01/e, 02/e, 03/e, 123/e}
            \draw (\from) -- (\to);
    \end{tikzpicture}
    \end{center}

\subsection{Minimal Building Set} \label{section:minimal}

Let $M^{\text{br}}$ be the broom matroid and $B_m=\{0,1,2,3,123\}$ the minimal building set of the broom matroid.
The three inclusion-maximal nested sets are $\{0,x,123\}$ where $x = 1,2,3$.

The set $\max(B_m)$ is equal to $\{0,123\}$.
Thus the Bergman fan of $(M^{\text{br}}, B_m)$ has lineality space 
$L_{B_m} = \spn_\mathbb{R}\{e_0, e_1 + e_2 + e_3\}$
and three maximal cones
\[
    \sigma_{1} = L_{B_m} + \text{cone}(e_1)\,, \qquad
    \sigma_{2} = L_{B_m} + \text{cone}(e_2)\,, \qquad
    \sigma_{3} = L_{B_m} + \text{cone}(e_3)\,.
\]
Because we have a two dimensional lineality space, our normal complex will live inside a two dimensional affine linear subspace of $\R^4$. In this example, our normal complex will live in the two-dimensional affine space where $x_0=3$ and $x_0+x_1+x_2+x_3=0$.
In order to construct the normal complex, we take the piecewise linear function $\varphi$ whose values on each cone of the Bergman fan are 
\begin{align*}
    \varphi(x) & = 3e_0^*+e_1^*-4e_2^* \qquad \text{ for } x \in \sigma_1\,,\\
    \varphi(x) & = 3e_0^*+e_2^*-4e_1^* \qquad \text{ for } x \in \sigma_2\,,\\
    \varphi(x) & = 3e_0^*+e_3^*-4e_1^* \qquad \text{ for } x \in \sigma_3\,.
\end{align*}
In light of \Cref{lem:c-is-enough}, $\varphi$ is the piecewise linear function associated to the weight vector $c\in \R^{B_m}$ such that $c_0=3, c_{123}=0$ and $c_1=c_2=c_3=1$.
The normal complex of $B_m$ defined by $c$ is shown in the following figure.
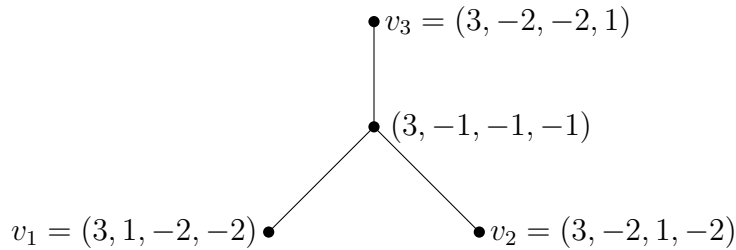
\begin{center}
\begin{tikzpicture}[scale=.7]
\coordinate (O) at (0,0);
\coordinate (1) at (-2,-2);
\coordinate (2) at (2,-2);
\coordinate (3) at (0,2);

\node[left] at (1) {$v_1=(3,1,-2,-2)$};
\node[right] at (2) {$v_2=(3,-2,1,-2)$};
\node[right] at (3) {$v_3=(3,-2,-2,1)$};
\node[right] at (O) {\,$(3,-1,-1,-1)$};

\foreach \x in {O,1,2,3} \node[fill=black,circle,inner sep=1.5] at (\x) {};
\foreach \x in {1,2,3} \draw (O) -- (\x);
\end{tikzpicture}
\end{center}

    We now give an example of the normal complex order $<_\cN$ (\Cref{def:normal-complex-order}). As guaranteed by \Cref{thm1}, the normal complex order is a shelling order of the nested set complex.
    Let $\gamma = (1000,100,10,1)$. This vector is lexicographic on $\cN_{B_m}$ (\Cref{def:lexicographic}).
    Taking the inner product of $\gamma$ with each of the three vertices corresponding to maximal nested sets gives
    \[
    \langle \gamma, v_{1} \rangle  = 3078 \qquad
    \langle \gamma, v_{2} \rangle  = 2808 \qquad
    \langle \gamma, v_{3} \rangle  = 2781
    \]
    This induces the following order on the maximal nested sets
    \[
    \{0,1,123\} <_\cN \{0,2,123\} <_\cN \{0,3,123\}\,,
    \]
    which is a shelling order of the nested set complex (shown below).
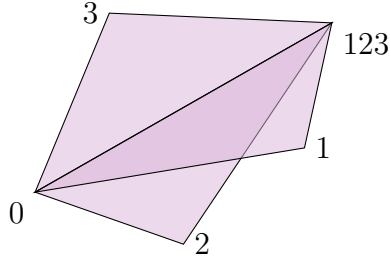
\begin{center}
     \tdplotsetmaincoords{60}{50}
    \begin{tikzpicture}[scale=.7, tdplot_main_coords]
    \coordinate (O) at (-2,-2,-2);
    \coordinate (123) at (2,2,2);
    \coordinate (1) at (0,3,-2);
    \coordinate (3) at (-1,-1,2);
    \coordinate (2) at (0, 0, -3);

    \node[below left] at (O) {$0$};
    \node[below right] at (123) {$123$};
    \node[left] at (3) {$3$};
    \node[right] at (2) {$2$};
    \node[right] at (1) {$1$};

    \draw[draw, fill=violet!30, fill opacity=.5] (O) -- (123) -- (2) -- cycle;
    \draw[draw, fill=violet!30, fill opacity=.5] (O) -- (123) -- (3) -- cycle;
    \draw[draw, fill=violet!30, fill opacity=.5] (O) -- (123) -- (1) -- cycle;
    \end{tikzpicture}
\end{center}

\subsection{Maximal Building Set} \label{example:maximal}

Let $M^{\text{br}}$ be the broom matroid and $B_M$ the maximal building set.
The inclusion-maximal nested sets are
\begin{align*}
\{0,01,0123\},& \{0,02,0123\}, \{0,03,0123\}, \{1,01,0123\}, \{2,02,0123\}, \\
\{3,03,0123\}, & \{1,123,0123\}, \{2,123,0123\}, \{3,123,0123\}\,.
\end{align*}
The unique inclusion-maximal element of $B_M$ is $0123$, so the lineality space of the Bergman fan is
\[
L_{B_M} = \spn_\mathbb{R}(e_0 + e_1 + e_2 + e_3)\,.
\]
The fan itself has $9$ maximal cones, corresponding to the $9$ maximal nested sets listed above. 
Let $\varphi$ be the piecewise linear function associated to the weight vector $c\in \R^{B_M}$ defined by $c_X = (4-|X|)\thinspace |X|$ for all $X\in B_M$.

Below, on the left, we draw a truncated piece of the Bergman fan (where the rays are labeled by the flats they correspond to). On the right, we show the normal complex $\cN_{B_M}$ associated to $\varphi$ with some of its vertices labeled.

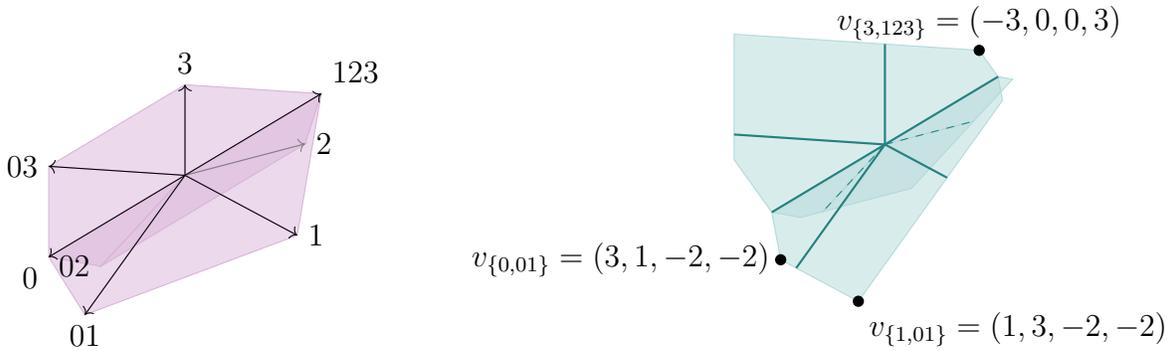
\begin{center}
\tdplotsetmaincoords{68}{55}
\begin{tikzpicture}[scale=1.3,tdplot_main_coords]

\coordinate (O) at (0,0,0);
\coordinate (A) at (2,0,0);
\coordinate (B) at (0,1.5,0);
\coordinate (C) at (0,0,1);
\coordinate (D) at (1,1,1);
\coordinate (E) at (0,-1.25,-1.25);
\coordinate (F) at (-1.5,0,-1.5);
\coordinate (G) at (-1,-1,0);
\coordinate (H) at (-1,-1,-1);

\draw[draw=violet!30,fill=violet!30,fill opacity=0.5] (O) -- (A) -- (D) -- cycle;
\draw[draw=violet!30,fill=violet!30,fill opacity=0.5] (O) -- (B) -- (D) -- cycle;
\draw[draw=violet!30,fill=violet!30,fill opacity=0.5] (O) -- (C) -- (D) -- cycle;
\draw[draw=violet!30,fill=violet!30,fill opacity=0.5] (O) -- (A) -- (E) -- cycle;
\draw[draw=violet!30,fill=violet!30,fill opacity=0.5] (O) -- (B) -- (F) -- cycle;
\draw[draw=violet!30,fill=violet!30,fill opacity=0.5] (O) -- (C) -- (G) -- cycle;
\draw[draw=violet!30,fill=violet!30,fill opacity=0.5] (O) -- (G) -- (H) -- cycle;
\draw[draw=violet!30,fill=violet!30,fill opacity=0.5] (O) -- (F) -- (H) -- cycle;
\draw[draw=violet!30,fill=violet!30,fill opacity=0.5] (O) -- (E) -- (H) -- cycle;

\draw[->] (O) -- (A);
\draw[->,gray] (O) -- (B); 
\draw[->] (O) -- (C);
\draw[->] (O) -- (D);
\draw[->] (O) -- (H);
\draw[->] (O) -- (G);
\draw[->] (O) -- (E); 

\node[right] at (A) {$1$};
\node[right] at (B) {$2$};
\node[above] at (C) {$3$};
\node[above right] at (D) {$123$};
\node[below] at (E) {$01$};
\node[left] at (F) {$02$};
\node[left] at (G) {$03$};
\node[below left] at (H) {$0$};

\end{tikzpicture}
\qquad 
\begin{tikzpicture}[scale=0.9,tdplot_main_coords]

\coordinate (O) at (0,0,0);
\coordinate (A) at (0,0,1.6);
\coordinate (B) at (1,1,1.6);
\coordinate (C) at (1.2,1.2,1.2);
\coordinate (D) at (0,1.6,0);
\coordinate (E) at (1,1.6,1);
\coordinate (F) at (1.6,0,0);
\coordinate (G) at (1.6,1,1);
\coordinate (H) at (-1.6,-1.6,1.6);
\coordinate (I) at (-1.6,-1.6,0);
\coordinate (J) at (-1.6,1.6,-1.6);
\coordinate (K) at (-1.6,0,-1.6);
\coordinate (L) at (1.6,-1.6,-1.6);
\coordinate (M) at (0,-1.6,-1.6);
\coordinate (N) at (-1.6,-1.6,-0.4);
\coordinate (P) at (-1.2,-1.2,-1.2);
\coordinate (Q) at (-1.6,-0.4,-1.6);
\coordinate (R) at (-0.4,-1.6,-1.6);

\draw[draw=teal!30, fill=teal!30, fill opacity=.5] (O) -- (A) -- (B) -- (C) -- cycle;
\draw[draw=teal!30, fill=teal!30, fill opacity=.5] (O) -- (D) -- (E) -- (C) -- cycle;
\draw[draw=teal!30, fill=teal!30, fill opacity=.5] (O) -- (F) -- (G) -- (C) -- cycle;
\draw[draw=teal!30, fill=teal!30, fill opacity=.5] (O) -- (A) -- (H) -- (I) -- cycle;
\draw[draw=teal!30, fill=teal!30, fill opacity=.5] (O) -- (D) -- (J) -- (K) -- cycle;
\draw[draw=teal!30, fill=teal!30, fill opacity=.5] (O) -- (F) -- (L) -- (M) -- cycle;
\draw[draw=teal!30, fill=teal!30, fill opacity=.5] (O) -- (I) -- (N) -- (P) -- cycle;
\draw[draw=teal!30, fill=teal!30, fill opacity=.5] (O) -- (K) -- (Q) -- (P) -- cycle;
\draw[draw=teal!30, fill=teal!30, fill opacity=.5] (O) -- (M) -- (R) -- (P) -- cycle;

\draw[thick,teal!70!gray] (O) -- (C);
\draw[thick,teal!70!gray] (O) -- (A);
\draw[dashed,teal!70!gray] (O) -- (D);
\draw[thick,teal!70!gray] (O) -- (F);
\draw[thick,teal!70!gray] (O) -- (M);
\draw[dashed,teal!70!gray] (O) -- (K);
\draw[thick,teal!70!gray] (O) -- (I);
\draw[thick,teal!70!gray] (O) -- (P);

\node[above] at (B) {$v_{\{3,123\}}=(-3,0,0,3) $};
\node[fill=black,inner sep=1.5,circle] at (B) {};
\node[fill=black,inner sep=1.5,circle] at (L) {};
\node[below right] at (L) {$v_{\{1,01\}}=(1,3,-2,-2)$};
\node[fill=black,inner sep=1.5,circle] at (R) {};
\node[left] at (R) {$v_{\{0,01\}}=(3,1,-2,-2)$};

\end{tikzpicture}

\end{center}
Let $\gamma = (1000,100,10,1)$. This vector is lexicographic on $\cN_{B_M}$ (\Cref{def:lexicographic}).
Taking the inner product of $\gamma$ with each of the nine vertices corresponding to the inclusion-maximal nested sets gives
\begin{align*}
\langle \gamma, v_{\{0,01,0123\}} \rangle &  = 3078  &
\langle \gamma, v_{\{0,02,0123\}} \rangle &  = 2808 \\ 
\langle \gamma, v_{\{0,03,0123\}} \rangle &  = 2781  & 
\langle \gamma, v_{\{1,01,0123\}} \rangle &  = 1278 \\ 
\langle \gamma, v_{\{2,02,0123\}} \rangle &  = 828  &
\langle \gamma, v_{\{3,03,0123\}} \rangle &  = 783 \\ 
\langle \gamma, v_{\{1,123,0123\}} \rangle &  = -2700 & 
\langle \gamma, v_{\{2,123,0123\}} \rangle &  = -2970 \\ 
\langle \gamma, v_{\{3,123,0123\}} \rangle &  = -2997 
\end{align*}
This induces the following order on the maximal nested sets
\begin{align*}
\{0,01,0123\} <_\cN & \{0,02,0123\} <_\cN \{0,03,0123\} <_\cN \{1,01,0123\} <_\cN \{2,02,0123\}\\
<_\cN & \{3,03,0123\} <_\cN \{1,123,0123\} <_\cN \{2,123,0123\} <_\cN \{3,123,0123\}\,,
\end{align*}
and we can check that this is a shelling order of the nested set complex.

The lexicographic (\Cref{def:lexicographic}) condition for $\gamma$ is necessary. If $\gamma'= (1,100,101,-1000)$ (this vector is not lexicographic), then the first two maximal nested sets in the induced order are $\{2,02,0123\} < \{1,01,0123\}$. We can already see that the induced order fails to be a shelling order because the first two facets of the nested set complex don't intersect!

Since $B_M$ is the maximal building set, its nested set complex is a cone over the order complex of the proper part of the lattice of flats.
Famously, Björner\cite{bjorner} gives a shelling order $\prec_{EL}$ (\Cref{def:el}) of the order complex of the lattice of flats with respect to any linear order on the ground set. 
Taking $0 \prec 1 \prec 2 \prec 3$ as our linear order of our ground set, we compute
\begin{align*}
\{0,01,0123\} \prec_{EL} & \{0,02,0123\} \prec_{EL} \{0,03,0123\} \prec_{EL} \{1,01,0123\} \prec_{EL} \{1,123,0123\} \\
\prec_{EL} & \{2,02,0123\}  \prec_{EL} \{2,123,0123\} \prec_{EL} \{3,03,0123\} \prec_{EL} \{3,123,0123\}\,.
\end{align*}
This is a different order than our shelling order coming from the normal complex, even though they have the same first element. 

\section{Combinatorics of Nested Set Complexes: Links and Joins}\label{sec:links}

Here we collect a series of technical combinatorial results about nested set complexes. 
These will be used in \Cref{subsec:geometry} to prove \Cref{thm1}.
First we start with some results that use poset theory to look at the topology of nested set complexes.
Then we construct a linear order on nested sets and prove some properties of this (combinatorial) order.

\begin{lemma}\label{lem:dim-nested-set}\label{lem:max-partition}
Let $M$ be a (loopless) matroid, $\cL$ its lattice of flats, and $B\subseteq \cL$ a building set of $\cL$ with containment-maximal elements $\max(B)$.
Then 
\begin{enumerate}
    \item Every inclusion-maximal nested set has cardinality equal to $\rank(M)$. The nested set complex $\Delta(\cL,B)$ is pure of dimension $\rank(M)- |\max(B)|$.
    \item The set $\max(B)$ partitions the ground set of $M$.
    \item Each flat $X\in B$ is contained in a unique element $F\in \max(B)$.
\end{enumerate}
\end{lemma} 

\begin{proof}
The first part follows from \cite[Corollary 4.3]{FM} after accounting for the fact that our definition of a nested set complex excludes $\max(B)$ from its vertex set.
The second part follows from Proposition \ref{prop:backman-danner} and the fact that the atoms of $M$ (which are contained in the set of join-irreducibles) partition the ground set.
We prove the last statement by contradiction.
Let $H_1,H_2\in \max(B)$, and
suppose $F \subseteq H_1, H_2$.
Then $H_1 \cap H_2 \neq \emptyset$, hence $H_1 \vee H_2 \in B$, a contradiction to the maximality of $H_1$ and $H_2$ via Proposition \ref{prop:backman-danner}.
\end{proof}

We now turn to understanding the link of a vertex in the nested set complex and joins of (very special) nested set complexes.
These results imbue nested set complexes with a recursive structure: we will be able to express the links of vertices in nested set complexes as a join of two smaller nested set complexes.

A surprising outcome is that this recursive structure is reflected in the geometry of the normal complex; see Lemma \ref{lem:intersection}.
This gives the normal complex itself a recursive structure mimicking the recursive structure of the combinatorial picture.
This recursive structure will be a key tool for the proof of Theorem \ref{thm1}.

Recall the definitions of $\cL\vert_X, \cL^X, B|_X$, and $B^X$ from earlier.
We will be interested in understanding the set $(B\vert_X,\emptyset)\cup (\emptyset, B^X)$ inside $\cL\vert_X \times \cL^X$.
To make the following exposition somewhat easier to parse, note that if $(Y,\emptyset)\in (B\vert_X,\emptyset)$, then $Y\subseteq X$. Similarly, if $(\emptyset,Z)\in (\emptyset, B^X)$, then $Z\subseteq E\setminus X$.
When describing elements of $(B\vert_X,\emptyset)\cup (\emptyset, B^X)$, we will often write $Y$ and $Z$ instead of $(Y,\emptyset)$ and $(\emptyset, Z)$, respectively.
Similarly, we will refer to the sets $(B\vert_X,\emptyset)$ and $(\emptyset, B^X)$ as $B\vert_X$ and $B^X$.
For simplicity, we will write $(B\vert_X \cup B^X)$ to mean $(B\vert_X,\emptyset)\cup (\emptyset, B^X)$.

\begin{lemma}\label{lem:building-set-union}
    The collection $B|_X \cup B^X$ is a building set in $\cL|_X \times \cL^X$.
\end{lemma}

\begin{proof}
Note that $\cL|_X$ and $\cL^X$ are both geometric lattices; see \cite[Sections 1.3 and 3.1]{oxley}.
In particular, their product is also a geometric lattice and it makes sense
to talk about building sets in $\cL|_X \times \cL^X$; see \cite[Fact 4.2.16]{oxley}.

It is well known that $B|_X$ is a building set in $\cL|_X$ and $B^X$ is a building set in $\cL^X$; see \cite[Lemma A.8]{wondertopes} or \cite[Proposition 2.40]{Mantovani-Pardol-Pilaud} for example.
We just need to check that their disjoint union is a building set for the product of posets $\cL|_X \times \cL^X$.
Following Proposition \ref{prop:backman-danner}, there are two things to check: first that the connected flats of $\cL|_X \times \cL^X$ are in $(B|_X, \emptyset) \cup (\emptyset, B^X)$ and second that whenever $(Y,Z) \wedge (V,W) \not= \emptyset$, their join is also in $(B|_X, \emptyset) \cup (\emptyset, B^X)$.

For the first statement, note that $\cL|_X \times \cL^X$ is the lattice of flats of $M|_X \oplus M^X$.
In particular, the connected flats have the form $(Y,\emptyset)$ or $(\emptyset,Z)$ and must already be in $(B|_X, \emptyset) \cup (\emptyset, B^X)$.
For the second statement, note that the only way for $(V,W), (Y,Z)\in (B|_X, \emptyset) \cup (\emptyset, B^X)$ to have $(V,W) \wedge (Y,Z)$ nonempty is for both of $(V,W), (Y,Z)$ to be in $(B|_X, \emptyset)$ or both to be in $(\emptyset, B^X)$.
In the first case, for example, we'd have $(V,\emptyset), (Y,\emptyset)$ with $V \wedge Y$ nonempty in $\cL|_X$.
Since $B|_X$ is a building set, $V\wedge Y \in B|_X$ (by Proposition \ref{prop:backman-danner}) and so $(V,\emptyset) \wedge (Y,\emptyset) \in (B|_X, \emptyset)$ too.
A similar argument holds when $(V,W), (Y,Z) \in (\emptyset, B^X)$, implying that their disjoint union is a building set.
\end{proof}

An interesting consequence of the preceding lemma is that the nested set complex $\Delta(\cL|_X \times \cL^X,B|_X\cup B^X)$ is isomorphic to the join of $\Delta(\cL|_X,B|_X)$ and $\Delta(\cL^X, B^X)$.
That is
\[\Delta(\cL|_X,B|_X) \ast \Delta(\cL^X, B^X) \cong \Delta(\cL|_X \times \cL^X,B|_X \cup B^X)\,.\]
In particular, there is a natural bijection between their vertex sets given by the map $(N,N') \mapsto N \cup N'$ (where $N$ sits inside the product poset by taking $(X,\emptyset)$ for each $X$ in $N$ and $N'$ sits inside the product by taking $(\emptyset,X)$ for each $X \in N'$).
It's easy to see that $ N \cup N'$ is nested whenever $N$ and $N'$ are, since $N \subseteq B|_X$ and $N' \subseteq B^X$ are on disjoint ground sets. 
In Lemma \ref{lem:intersection}, we will see the facets of a normal complex of $(\cL,B)$ are precisely the normal complexes of these joins.

We now recall another isomorphism of simplicial complexes, this time between the link of a flat and the product of the nested set complexes of the restriction and contraction at that flat.
This bijection was first given for Boolean matroids in \cite{zelevinsky} and then extended to all matroids in \cite[Theorem 1.7]{wondertopes}.
This map is first defined between the sets of vertices of the two complexes, then it can be shown that this map between vertex sets induces an isomorphism of simplicial complexes.

Let $Z \in B \setminus \max(B)$ be a flat of the building set which is not maximal and let ${\sf LinkVert}(Z)$ denote the set of vertices in the link of $Z$ , i.e.,
\begin{align*}
    {\sf LinkVert}(Z) & = 
     \{X\in B\setminus (\{Z\}\cup\max(B)): \{X,Z\} \text{ is reduced nested}\}.
\end{align*}
Now the map $\tau_Z$ is
\begin{align*}
    \tau_Z: {\sf LinkVert}(Z) & \to (B\vert_Z \cup B^Z) \setminus \max(B\vert_Z \cup B^Z)\\
    X & \mapsto \begin{cases}
X & \text{if } X < Z\\
(X \vee Z) \setminus Z & \text{else.}
\end{cases}
\end{align*}

What makes $\tau_Z$ particularly useful is the following theorem, which is inspired (both in statement and proof style) by work of Zelevinsky for the Boolean lattice \cite{zelevinsky}.

\begin{theorem}[\!\!{{\cite[Theorem 1.7]{wondertopes}}}]\label{thrm:wondertopes}
    The map $\tau_Z$ is a bijection and induces an isomorphism between the simplicial complexes ${\sf Link}(\Delta(\cL, B);Z)$ and $\Delta(\cL\vert_Z \times \cL^Z, B\vert_Z \cup B^Z)$.
\end{theorem}

\begin{note} \label{note:maximalTau}
The maximal elements of $\max(B\vert_Z \cup B^Z)$ are $\{Z\} \cup \max(B)$ and the same construction as $\tau_Z$ also gives a bijection from $\{Z\} \cup \max(B)$ to $\max(B\vert_Z \cup B^Z)$.
It will sometimes be useful for us to use this ``extension" of $\tau_Z$ beyond the setting of links.
There, the map doesn't have any meaning in terms of simplicial complexes, but we will use it later in some of our arguments.
\end{note}

\begin{remark}[Comparing Notation]\label{rem:wondertopes-conventions}
Our notation differs slightly from the notation in \cite{wondertopes}.
First, our definition of matroid contraction differs slightly from \cite[Definition A.6]{wondertopes}.
We delete elements when we contract, so our map $\tau_Z$ looks slightly different.
Second, the authors of \cite{wondertopes} use $B|_X \times B^X$ to denote $(B|_X,\emptyset) \cup (\emptyset, B^X)$.
While this makes it easier to understand that this building set sits inside the product of the restriction and deletion poset, we have opted to use the more technically-accurate (although more cumbersome) notation of unions; see Lemma \ref{lem:building-set-union}.
\end{remark}

\begin{example}\label{ex:tau-map1}
Consider the following geometric lattice with maximal building set $B$ (circled).
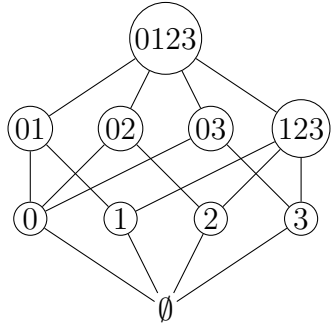
\begin{center}
      \begin{tikzpicture}[scale=.6,inner sep=1] 
        \node (em) at (0, 0) {$\emptyset$};
        \node[draw,circle] (0) at (-3,2) {$0$};
        \node[draw,circle] (1) at (-1,2) {$1$};
        \node[draw,circle] (2) at (1, 2) {$2$};
        \node[draw,circle] (3) at (3, 2) {$3$};
        \node[draw,circle] (01) at (-3,4) {$01$};
        \node[draw,circle] (02) at (-1,4) {$02$};
        \node[draw,circle] (03) at (1, 4) {$03$};
        \node[draw,circle] (123) at (3, 4) {$123$};
        \node[draw,circle] (e) at (0, 6) {$0123$};
          \foreach \from/\to in {em/0, em/1, em/2, em/3, 0/01, 0/02, 0/03, 1/01, 1/123, 2/02, 2/123, 3/03, 3/123, 01/e, 02/e, 03/e, 123/e}
            \draw (\from) -- (\to);
    \end{tikzpicture}
\end{center}
The vertices of the link of $0$ are $\{01,02, 03, 0123\}$ and the restriction and contraction of $B$ at $X$ are
\[
    B|_X = \{0\} \qquad \text{and} \qquad
    B^X = \{01,02,03,0123\}\,.
\]
The $\tau_0$ map sends the vertices of the link of $0$ to $B|_X \cup B^X$ without the maximal elements $\{0, 0123\}$.
That is, 
\begin{align*}
    \tau_0(01) = (01 \vee 0)\setminus 0 = 1\\
    \tau_0(02) = (02 \vee 0)\setminus 0 = 2\\
    \tau_0(03) = (03 \vee 0)\setminus 0 = 3\,.
\end{align*}
On the level of posets, the image of $\tau_X$ sits inside the product poset $\cL|_X \times \cL^X$.
Below we draw $\cL|_X \times \cL^X$ and a poset which is isomorphic to this product under the map $(Y,Z) \mapsto Y\cup Z$.
To illustrate how the building sets sit inside this lattice, we circle $B|_X$ (in red) and $B^X$ (in blue) along with their images under this map.

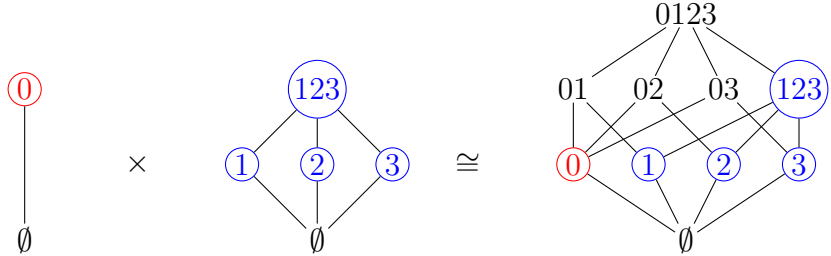
\begin{center}
\begin{tikzpicture}[scale=.5,inner sep=1] 
    \node (em) at (0, 0) {$\emptyset$};
    \node[draw, circle,red] (0) at (0,4) {$0$};
      \foreach \from/\to in {em/0}
        \draw (\from) -- (\to);
    \node[] (times) at (3, 2) {$\times$};
\end{tikzpicture}
\qquad
\begin{tikzpicture}[scale=.5,inner sep=1] 
    \node (em) at (0, 0) {$\emptyset$};
    \node[draw, circle,blue] (1) at (-2,2) {$1$};
    \node[draw, circle,blue] (2) at (0, 2) {$2$};
    \node[draw, circle,blue] (3) at (2, 2) {$3$};
    \node[draw, circle,blue] (123) at (0, 4) {$123$};
      \foreach \from/\to in {em/1, em/2, em/3, 1/123, 2/123, 3/123}
        \draw (\from) -- (\to);

    \node[] (equals) at (4, 2) {$\cong$};
\end{tikzpicture}
\qquad
\begin{tikzpicture}[scale=.5,inner sep=1] 
    \node (em) at (0, 0) {$\emptyset$};
    \node[draw, circle,red] (0) at (-3,2) {$0$};
    \node[draw, circle,blue] (1) at (-1,2) {$1$};
    \node[draw, circle,blue] (2) at (1, 2) {$2$};
    \node[draw, circle,blue] (3) at (3, 2) {$3$};
    \node (01) at (-3,4) {$01$};
    \node (02) at (-1,4) {$02$};
    \node (03) at (1, 4) {$03$};
    \node[draw, circle,blue] (123) at (3, 4) {$123$};
    \node (e) at (0, 6) {$0123$};
      \foreach \from/\to in {em/0, em/1, em/2, em/3, 0/01, 0/02, 0/03, 1/01, 1/123, 2/02, 2/123, 3/03, 3/123, 01/e, 02/e, 03/e, 123/e}
        \draw (\from) -- (\to);
\end{tikzpicture}
\end{center}
Similarly, the map $\tau_1$ is a bijection between 
\[
\{01,123\} \qquad \text{and} \qquad \emptyset \cup \{01,23\}\,.
\]
Concretely, it is defined by
\begin{align*}
    \tau_1(01) & = (01 \vee 1) \setminus 1 = 0\\
    \tau_1(123) & = (123 \vee 1) \setminus 1 = 23\,.
\end{align*}
On the level of posets, $\tau_1$ sends some flats of the original poset $\cL$ to some flats of $\cL|_X \times \cL^X$.
The product poset now looks slightly different than it did for $\tau_0$.
Now we have (where again, the poset on the right is the image of the product poset under the map $(Y,Z) \mapsto Y\cup Z)$
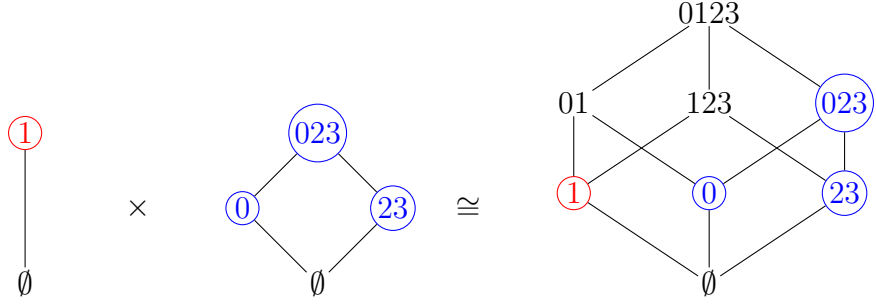
\begin{center}
\begin{tikzpicture}[scale=.5,inner sep=1] 
    \node (em) at (0, 0) {$\emptyset$};
    \node[draw, circle,red] (0) at (0,4) {$1$};
      \foreach \from/\to in {em/0}
        \draw (\from) -- (\to);
    \node[] (times) at (3, 2) {$\times$};
\end{tikzpicture}
\qquad
\begin{tikzpicture}[scale=.5,inner sep=1] 
    \node (em) at (0, 0) {$\emptyset$};
    \node[draw, circle,blue] (0) at (-2,2) {$0$};
    \node[draw, circle,blue] (23) at (2, 2) {$23$};
    \node[draw, circle,blue] (023) at (0, 4) {$023$};
      \foreach \from/\to in {em/0, em/23, 0/023, 23/023}
        \draw (\from) -- (\to);

    \node[] (equals) at (4, 2) {$\cong$};
\end{tikzpicture}
\qquad
      \begin{tikzpicture}[scale=.6,inner sep=1] 
        \node (em) at (0, 0) {$\emptyset$};
        \node[draw,circle,red] (1) at (-3,2) {$1$};
        \node[draw,circle,blue] (0) at (0,2) {$0$};
        \node[draw,circle,blue] (23) at (3, 2) {$23$};
        \node[draw,circle,blue] (023) at (3, 4) {$023$};
        \node (10) at (-3,4) {$01$};
        \node (123) at (0, 4) {$123$};
        \node (1023) at (0, 6) {$0123$};
          \foreach \from/\to in {em/0, em/1, em/23, 1/10,1/123,0/10,0/023,23/123,23/023,023/1023,10/1023,123/1023}
            \draw (\from) -- (\to);
    \end{tikzpicture}
\end{center}
\end{example}

Remarkably, the combinatorics of $\tau_Z$ is also visible in the geometry of normal complexes.
On the geometric side, taking a link corresponds to intersecting with a facet-defining hyperplane.
In the following lemma, we show that this intersection is a normal complex of the direct sum of the restriction and contraction along $Z$.
We will find that the combinatorial isomorphism $\tau_Z$ translates to a pointwise equality on the level of normal complexes.

\begin{lemma}\label{lem:intersection}
    Let $\cN$ be a normal complex of $(\cL,B)$, $Z$ a flat in $\cL$, and $H_{Z,\varphi}$ one of the defining hyperplanes of $\cN$.
    Then 
    \begin{enumerate}
        \item $\cN \cap H_{Z,\varphi}$ is a normal complex $\cN_Z$ of $(\cL\vert_Z \times \cL^Z, B\vert_Z \sqcup B^Z)$ whose dimension is one less than the dimension of $\cN$, and 
        \item For any inclusion-maximal nested set $N$ containing $Z$, the vertex $v_N$ of $\cN$ is equal (as a point inside of $\R^E$) to the vertex $v_{\tau_Z(N)}$ of the normal complex $\cN_Z$.
    \end{enumerate}
\end{lemma}

\begin{proof}
For the first statement, the fact that $\cN \cap H_{Z,\varphi}$ is a polytopal complex of one dimension less follows from the fact that $H_{Z,\varphi}$ is a facet-defining hyperplane; see Definition \ref{def:facets}.
The tricky part is to check that the intersection is a normal complex of the direct sum of the contraction and restriction. 
To see that $\cN \cap H_{Z,\varphi}$ is the normal complex of $(\cL\vert_Z \times \cL^Z, B\vert_Z \sqcup B^Z)$, we will use the fact that the nested set complex of $(\cL\vert_Z \times \cL^Z, B\vert_Z \sqcup B^Z)$ is isomorphic to the link of $Z$ in the nested set complex of $(\cL,B)$.
We will use the correspondence between the cones of their Bergman fans (here: Bergman fan of a link is used informally to mean the cones of the Bergman fan of $(\cL,B)$ that correspond to nested sets of the link of $Z$).

Let $c \in \mathbb{R}^B$ be the vector defining the cubical function $\varphi$ for $\cN$ as described in Lemma \ref{lem:c-is-enough}. Define $\tau_Z(c)\in \mathbb{R}^{B\vert_Z \sqcup B^Z}$ by declaring \[
\tau_Z(c)_{\tau_Z(X)} = \begin{cases}
c_X- c_Z,& Z<X\\
c_X,& \text{else.}
\end{cases}\]
Using the notation from Note \ref{note:hxphi-hxc}, we first prove that \[H_{\tau_Z(X),\tau_Z(c)} \cap H_{\tau_Z(Z),\tau_Z(c)}= H_{X,\varphi} \cap H_{Z,\varphi}\] for all $X\neq Z$ such that $\{X,Z\}\cup \max(B)$ is nested. Once we prove this, the only thing remaining to check for the first part of our claim is that $\tau_Z(c)$ is cubical. That is, $v_{\tau_Z(N)}$ is in the relative interior of $\sigma_{\tau_Z(N)}$.

Let $X$ be a flat of $B$, not equal to $Z$, such that $\{X,Z\} \cup \max(B)$ is nested. If $Z< X$, then $\tau_Z(X) = X\setminus Z$. In this case, \begin{align*}H_{\tau_Z(X),\tau_Z(c)} \cap H_{\tau_Z(Z),\tau_Z(c)} &= \{v\in \R^E: \langle v, e_{X\setminus Z} \rangle  = c_X-c_Z \text{ and } \langle v, e_{Z} \rangle = c_Z\}\\
&= \{v\in \R^E: \langle v, e_{X} \rangle  = c_X \text{ and } \langle v, e_{Z} \rangle = c_Z\}\\
&= H_{X,\varphi} \cap H_{Z,\varphi}.
\end{align*}
If $X$ and $Z$ are incomparable or $X< Z$ then, by \Cref{cor:incomp2}, $\tau_Z(X)=X$. In this case, $H_{\tau_Z(X), \tau_Z(c)}$ is literally equal to $H_{X,\varphi}$.

Now given a nested set $\tau_Z(N)$ of $(\cL\vert_Z \times \cL^Z, B\vert_Z \sqcup B^Z)$, we need to check that the point $v_{\tau_Z(N)} = \bigcap_{\tau_Z(X)\in \tau_Z(N)} H_{\tau_Z(X),\tau_Z(c)}$ is contained in the relative interior of $\sigma_{\tau_Z(N)}$. The calculations in the previous paragraph imply that $v_{\tau_Z(N)}= v_N$. As $\varphi$ is cubical, $v_N$ sits inside the relative interior of $\sigma_N$. Thus our claim follows once we check that $\sigma_N \subseteq \sigma_{\tau_Z(N)}$. This, in turn, can be checked by confirming that the set $\{e_X: X\in N\setminus \max(B)\} \cup \{e_Y,-e_Y: Y\in \max(B)\}$ is contained in $\sigma_{\tau_Z(N)}$. However, this check is immediate from the definition of $\tau_Z$ used in the previous paragraph.
This proves the first part of the claim.

The second part of the claim follows immediately from the preceding paragraph.
\end{proof}

\begin{proposition}\label{prop:lex_link}
    Let $\cN$ be a normal complex of $(\cL,B)$, $\gamma$ a lexicographic vector on $\cN$, $Z$ a flat of $B\setminus \max(B)$ and $\cN_Z$ be the normal complex of $(\cL\vert_Z \times \cL^Z, B\vert_Z \cup B^Z)$ described in \Cref{lem:intersection}. Also let $\tau_Z$ be the bijection between the nested sets of $(\cL,B)$ containing $Z$ and the nested sets of $(\cL\vert_Z \times \cL^Z, B\vert_Z \cup B^Z)$ defined in \Cref{sec:links}. Then
    \begin{itemize}
        \item $\gamma$ is lexicographic on $\cN_Z$, and
        \item for maximal nested sets $N$ and $N'$ containing $Z$, we have
    \[N<_{\cN} N' \qquad \iff \qquad \tau_Z(N) <_{\cN} \tau_Z(N').\]
    \end{itemize} 
\end{proposition}
\begin{proof}
    Both claims follow immediately from the second part of \Cref{lem:intersection}. 
\end{proof}

\begin{remark}
Proposition \ref{prop:lex_link} relies on the geometry of the normal complex order.
The same statement often does not hold for the NL-order.
\end{remark}

\section{Proof of \Cref{thm1}}\label{sec:proof-of-main}

The goal of this section is to prove \Cref{thm1} by giving a shelling order on the Bergman fan of a matroid.
The proof is somewhat involved, and requires several definitions and preparatory results. 

In \Cref{subsec:combo_order}, we define a linear order on the facets of a nested set complex which we call the NL-order.
This linear order generalizes Björner's classical EL-shelling order; the two coincide in the case of the maximal building set.
The proof of our main theorem comes from leveraging the combinatorial properties of the NL-order against the geometry of normal complexes. 
In Subsection \ref{subsec:forestry}, we develop properties of this combinatorial order and its connection to the map $\tau_Z$ defined in Section \ref{sec:links}.
In \Cref{subsec:geometry}, we put everything together to prove \Cref{thm1}.

\subsection{A Combinatorial Order on the Nested Set Complex} \label{subsec:combo_order}
Let $M$ be a loopless matroid of rank $r$ and $B$ a building set of $\cL=\cL(M)$. 
Suppose that the set of atoms $\cL_1$ of $M$ is totally ordered to agree with a total ordering of the ground set of $M$. That is, for atoms $A$ and $A'$, we say that $A<A'$ if $\min(A)<\min(A')$. 

For a nested set $N$, consider the labeling $m_N: N \to \cL_1$ defined by 
\[m_N(X) = \min\left\{A \in \cL_1: A \leq X , A \not\leq Y \text{ for all } Y \in N \text{ with } Y< X \right\}\,. \]
Note that when $M$ is simple, each atom in $\cL_1$ is a singleton. 
We have to be slightly more careful here, as some of our inductive arguments require that we pass to a matroid contraction, and contracting a flat can produce parallel elements.\footnote{For example, if we take the matroid $U_{3,2}$ (the uniform matroid on three elements of rank $2$) and contract a single element, we obtain $U_{2,1}$ (the uniform matroid on two elements of rank $1$).}

The following lemma establishes that this labeling map is actually a sensible thing to define.

\begin{lemma}\label{lem:basic_NL}
    Let $N$ be a nested set of $(\cL,B)$. The labeling function $m_N: N\to \cL_1$ is well defined, and injective.

\end{lemma}
\begin{proof}
    To see that $m_N$ is well defined, suppose instead that all of the atoms below some flat $X$ were all contained in a union of flats $Y_1,\ldots, Y_k$ with $Y_i\in N$ and $Y_i\lneq X$. Without loss of generality, we can assume that the $Y_i$ form an antichain. But now $\bigvee_{i=1}^k Y_i = X$ and the definition of a nested set implies $X\not\in B$, which is a contradiction. 

    To see the injectivity of $m_N$, suppose that instead $m_N(X)=A=m_N(X')$ for two different flats $X$ and $X'$ of $N$. By the definition of $m_N$, this could only be possible if $X$ and $X'$ were incomparable. But we already know $A\subseteq X\cap X'$, so by \Cref{cor:disjoint} this is impossible. 
\end{proof}

We are now ready to define a combinatorial order on the facets of $\Delta(\cL,B)$, which we call the nested lexicographic order.  
The nested lexicographic order is different from our geometric normal complex order, but the two are related, and we utilize this relationship in the proof of Theorem \ref{thm1}.

The intuition for the nested lexicographic order is best understood through the perspective of Observation \ref{obs:forest}.
From this perspective, we can think of $m_N$ labeling the vertices of the forest of $N$.
The label of $X$ is the smallest atom below $X$ that does not also sit below any of the children of $X$ (in the forest of $N$).
There is a natural way to list the values $\{m_N(X) \mid X \in N\}$ by ``plucking'' $m_N$-minimal leaves of this forest and recording their $m_N$ values.
We make this precise now.

Given a nested set $N$, let $\min{}_*(N)$ be the inclusion-minimal flat $X$ of $N$ that has the smallest labeling $m_N(X)$ among all of the inclusion-minimal flats of $N$ (in the language of Observation \ref{obs:forest}, this is the leaf with the smallest $m_N$-label).
Given a nested set $N$, consider the ordering $X_1<X_2<\ldots<X_k$ of the flats of $N$ defined recursively by 
\begin{equation}\label{eq:internal_order}
X_i = \min{}_*(N\setminus \{X_1,\ldots, X_{i-1}\}). 
\end{equation}
In the language of Observation \ref{obs:forest}, $\min{}_*(N\setminus \{X_1,\ldots, X_{i-1}\})$ is the leaf with smallest $m_N$-label once we have removed $X_1,\ldots, X_{i-1}$ from $N$.

The \Dfn{NL-labeling} (nested lexicographic labeling) of the nested sets of $(\cL,B)$ is the function which assigns to each nested set $N= \{X_1,\ldots, X_k\}$ (where $X_1<X_2<\ldots<X_k$ in ordering defined by \Cref{eq:internal_order}) the sequence of distinct atoms 
\[m(N)\coloneqq (m_N(X_1), m_N(X_2),\ldots, m_N(X_k)). \]
The \Dfn{NL-ordering} (nested lexicographic ordering) 
is the linear order of the facets of $\Delta(\cL,B)$ induced by the lexicographic order on the NL-labelings of the facets; see Definition \ref{def:usual-lex-order}.  If $N,N'$ are maximal nested sets such that $N$ comes before $N'$ in the NL-order, we may write $N \prec_{NL} N'$.

\begin{remark}\label{ELremark}
As noted above, if $B$ is the maximal building set, the NL-order specializes to Bj\"orner's EL-order on the maximal chains in the lattice of flats.  For this claim to be technically correct, we must first perform a very slight modification of the EL-order.  The name EL-order stands for \emph{edge-lexicographic order} as this order is induced by a labeling of the edges of the Hasse diagram.  This labeling restricted to a fixed chain can be \emph{shifted a half step upwards} to induce a labeling of the vertices of this chain.  Once this shifting has been performed, the order agrees precisely with the NL-order for the order complex.  We note that, unlike the EL-order, the NL-order appears to only be definable via a vertex labeling and not via an edge labeling.    
\end{remark}

\begin{remark}\label{recursiveELremark}
We explain here how the EL-shelling of the order complex of the lattice of flats of a matroid can be verified in an inductive way which parallels the proof of the line shelling order of a projective fan (see the proof sketch for Theorem \ref{lineshellingtheorem}).

To begin, we naturally extend Bj\"orner's EL-order $\prec_{EL}$ from the order complex of the lattice of flats to joins of such complexes.  We describe here the order in the case of the join of an ordered pair of order complexes, and leave it to the reader to extend this notion to the join of a $k$-tuple of order complexes.  If $\Delta_1, \Delta_2$ are order complexes of reduced lattices of flats, i.e., with bottom and top elements of the lattices removed, and $(C, D),(C',D') \in \Delta_1 * \Delta_2$, we define $(C, D)\prec_{EL} (C',D')$ if $C \prec_{EL}C'$, or $C = C'$ and $D \prec_{EL} D'$.

For inductively verifying the EL-order is a shelling, we must work in greater generality of joins of $k$-tuples of order complexes of reduced lattices of flats of a matroid.  Let $\Delta$ be such a complex.  We proceed by induction on the dimension of $\Delta$ where the base case is a point. Next, suppose that $C,C'$ are facets of $\Delta$ with $C \prec C'$.  We wish to find a facet $C''$ such that $C \cap C' \subseteq C''\cap C' = C' \setminus \{X\}$ for some flat $X$.  Case 1: $C\cap C' = \emptyset.$  In this case, because $C'$ is not the minimum facet, Bj\"orner shows that there is a descent at a flat $X$ in a component chain in $C'$, and performing a swap at $X$ produces the desired facet $C''$.  Case 2:  There exists some $Z \in C \cap C'$.  We look at the image of $C$ and $C'$ in the link  of $Z$.  The EL-order on $\Delta$ restricts to the corresponding EL-order on the link of $Z$, and the desired result follows by induction on the dimension of the simplicial complex.

We make a few observations about the above argument.  The join of order complexes is a nested set complex for a certain building set on the direct sum of underlying matroids, and the NL-order on this nested set complex generalizes the order described above; it is possible, depending on the order of the ground set, for $(C, D)\prec_{NL} (C',D')$ even though $C' \prec_{NL} C$.  

The argument above works because if $C\prec_{EL}C'$, then this relative order is  preserved for their images in the link of the flat $Z$.  This is not so for the NL-order! (See example \ref{localtauexample}.) This obstacle prevents one from proving in such a straightforward inductive way that the NL-order is a shelling order.  On the other hand, the normal complex order does have this desired property and this is what allows us to provide an inductive proof that the normal complex order is a shelling order.
\end{remark}

The following proposition ensures that the NL-ordering is indeed a linear ordering.

\begin{lemma}\label{lem:nl_unique}
    Every inclusion-maximal set $N$ is uniquely reconstructible from its NL-labeling. In particular, if $N'$ and $N''$ are two inclusion-maximal nested sets such that $m(N')=m(N'')$, then $N'=N''$.
\end{lemma}
\begin{proof}
    We only need to prove the first statement. Suppose that we are given a tuple of atoms $(A_1,\ldots,A_r)$ and we know that there exists a nested set $N=\{Y_1,\ldots, Y_r\}$ such that 
    \[m(N) = (m_{N}(Y_1),m_N(Y_2),\ldots, m_N(Y_r)) = (A_1,\ldots, A_r)\,. \]
    We now give a recursive recipe to construct a nested set $\tilde N = \{X_1,\ldots,X_r\}$. We will then inductively prove that $Y_i=X_i$, thereby proving the claim. 

    Let $X_1= A_1$. For $k\geq 2$, recursively define $X_k$ to be the unique inclusion-maximal element of $B$ such that
    \[    A_k \leq X_k \leq A_k \vee \bigvee_{i=1}^{k-1} X_i.\]
    Note that this construction is well-defined and does not produce any duplicate flats. Indeed, as $(A_1,\ldots, A_r)$ is the NL-labeling of an inclusion-maximal nested set, we know that $A_k \not\leq \bigvee_{i=1}^{k-1} X_i$ for all $2\leq k \leq r$.

    We now make two observations about $\tilde N$. First, $\tilde N$ is a nested set.
    To see this, consider an antichain $\{X_{i_1},\ldots, X_{i_\ell}\}$ of flats of $\tilde N$ with $i_1<i_2<\ldots< i_\ell$ and $\ell\geq 2$. We will show that $Z= \bigvee_{j=1}^\ell X_{i_j} \not\in B$. By definition, 
    \[A_{i_\ell} \leq Z \leq \bigvee_{j=1}^{i_\ell} X_j \leq A_{i_\ell} \vee \bigvee_{j=1}^{i_{\ell}-1} X_j. \]
    Our construction chose $X_{i_{\ell}}$ to be the inclusion-maximal element of $B$ satisfying this condition. However, we know that $X_{i_\ell} <Z$, thus $Z\not\in B$.

    Our second observation is that, for all $k\geq 1$, the rank of $\bigvee_{i=1}^k X_i$ is equal to $k$ and $\bigvee_{i=1}^k X_i= \bigvee_{i=1}^k A_i$. To see this, note that
     $\rank(\bigvee_{i=1}^k X_i)  > \rank(\bigvee_{i=1}^{k-1} X_i)$ and joining an atom can increase the rank by at most 1, i.e. $\rank(\bigvee_{i=1}^k X_i)  \leq \rank(\bigvee_{i=1}^{k-1} X_i) +1$.

    We now prove by induction on $k$ that $Y_k=X_k$. Before beginning our induction, we note that if $Y_i \leq Y_k$ then $i\leq k$. 
    Combining \Cref{thrm:wondertopes} and the first part of \Cref{lem:dim-nested-set}, this observation implies 
    \begin{equation}\label{eq:rank}\tag{$\ast$}
    |\{Y_i: i\leq k, Y_i \leq Y_k\}| = \rank(Y_k)\,.
    \end{equation}
    We are now ready to start the inductive argument.
        \Cref{eq:rank} implies that $\rank(Y_1)=1$ so $Y_1=A_1=X_1$ and our claim is true in the base case.
        Now assume $k \geq 2$ and that $Y_i=X_i$ for all $i<k$. We now aim to prove that $Y_k=X_k$. We do so by showing that we cannot have $Y_k>X_k$, $Y_k<X_k$, or $Y_k,X_k$ being incomparable.
        \begin{enumerate}
            \item {\sf Case 1}: $Y_k > X_k$. As a consequence of \Cref{eq:rank}, the set $\{Y_k\} \cup \{X_i: i<k, X_i< Y_k \}$ is an inclusion-maximal nested set of $(\cL\vert_{Y_k}, B\vert_{Y_k})$.
            Since $X_k< Y_k$ and $\tilde N$ is a nested set, we also find that
            \[\{Y_k, X_k \} \cup \{X_i: i<k, X_i< Y_k \}\]
            is a nested set of $(\cL\vert_{Y_k}, B\vert_{Y_k})$.
            This is a contradiction.
            \item {\sf Case 2}: $Y_k<X_k$. An identical argument to the first case yields the claim.
            \item {\sf Case 3}: $Y_k$ and $X_k$ are incomparable. As $Y_k$ contains $A_k$, $Y_k$ is an inclusion maximal flat of the set $\{X_1,\ldots, X_{k-1}, Y_k\}$. By \Cref{lem:incomp}, this implies that $Y_k$ is an inclusion maximal element of $B$ contained in the interval $[\emptyset, Y_k \vee \bigvee_{i=1}^{k-1} X_i]$. As $Y_k$ contains $A_k$, we know that $X_k$ is contained in the interval $[\emptyset, Y_k \vee \bigvee_{i=1}^{k-1} X_i]$. Let $W$ be the inclusion maximal element of $B$ such that $X_k \leq W \leq Y_k \vee \bigvee_{i=1}^{k-1} X_i$. As $X_k$ and $Y_k$ are incomparable, we know that $W\neq Y_k$. However $A_k \subseteq W\cap Y_k$, which contradicts the second part of \Cref{lem:dim-nested-set}.
    \end{enumerate}
     
\end{proof}
\begin{remark}
Although the NL-order is combinatorially defined, it does have some geometric content when viewed through the lens of the Bergman fan.  Let $N$ be a maximal nested set and $<$ a linear extension of the associated forest poset.  Let $X^<_i$ be an ordering of the flats of $N$ according to $<$, and let $Y^<_k = \vee_{i=1}^kX^<_i$ so that the $\{Y^<_k\}$ gives a maximal chain of flats.  The chamber associated to $N$ in the Bergman fan induced by the building set $B$ is triangulated by the cones associated to the chain $\{Y^<_k\}$ as $<$ varies over all linear extensions of the poset.
We can reinterpret the NL-labeling of $N$ as the lexicographically minimal EL-labeling of the chains $\{Y^<_k\}$ as we vary over all such linear extensions $<$.  This has some consequences.  For example, this immediately gives an a geometric proof of Lemma \ref{lem:nl_unique}:  since we know that the EL-labeling for the order complex is injective, and each chamber of the fine Bergman fan is contained in a unique chamber of the Bergman fan induced by $B$, it follows from the above remarks that the NL-order is also injective.
\end{remark}

We say that the NL-labeling $m(N)= (m_N(X_1),m_N(X_2),\ldots,m_N(X_k))$ of a nested set $N$ has a \Dfn{descent} at $X_i$ if $m_N(X_i)>m_N(X_{i+1})$. We declare the NL-labeling $m(N)$ \Dfn{increasing} if $m(N)$ has no descents. The following proposition is an analogue of one of the key properties of EL-labelings \cite[Lemma 7.6.2]{bjorner}.

\begin{proposition}\label{prop:new_NL_min}
There exists a unique inclusion-maximal nested set $N_{\min}$ such that $m(N_{\min})$ is increasing. Furthermore, $N_{\min}$ is the minimum nested set of the NL-ordering.
\end{proposition}
\begin{proof}
        We first construct the nested set $N_{\min}=\{X_1,\ldots,X_r\}$. 
        Let $A_1$ be the minimum atom of $\cL_1$ and set $X_1=A_1$. Recursively define the atom $A_k$ and the flat $X_k$ by setting
    \[A_k = \min\{A\in \cL_1: A\not\leq \bigvee_{i=1}^{k-1} X_i \}\]
    and $X_k$ to be the unique inclusion-maximal flat of $B$ such that $A_k\leq X_k \leq \left(A_k \vee \bigvee_{i=1}^{k-1} X_i\right)$. Note that this is well construction is well defined. Indeed, for $k \leq r$, 
    \[\rank(\bigvee_{i=1}^{k-1}X_i) = \rank(\bigvee_{i=1}^{k-1}A_i) \leq k-1 < r \]
    and thus the set $\{A\in \cL_1: A\not\leq \bigvee_{i=1}^{k-1} X_i \}$ is always nonempty. 
    Also note that our construction of the $A_k$ ensures that $A_1<A_2<\cdots <A_r$.

    We now show that $N_{\min}$ is an inclusion-maximal nested set. As all every nested set of size $r$ is inclusion-maximal (\Cref{lem:dim-nested-set}), we only need to check that $N_{\min}$ is nested.
    Consider an antichain $\{X_{i_1},\ldots, X_{i_\ell}\}$ of flats of $N_{\min}$ with $i_1<i_2<\ldots< i_\ell$ and $\ell\geq 2$. We will show that $Y= \bigvee_{j=1}^\ell X_{i_j} \not\in B$. By definition, 
    \[A_{i_\ell} \leq Y \leq \bigvee_{j=1}^{i_\ell} X_j \leq A_{i_\ell} \vee \bigvee_{j=1}^{i_{\ell}-1} X_j. \]
    Our construction chose $X_{i_{\ell}}$ to be the inclusion-maximal element of $B$ satisfying this condition. However, we know that $X_{i_\ell} <Y$, thus $Y\not\in B$. 
    From the way that $N_{\min}$ is constructed, we have
    \[m(N_{\min})= (m_{N_{\min}}(X_1), m_{N_{\min}}(X_2), \ldots , m_{N_{\min}}(X_r)) = (A_1,A_2,\ldots, A_k). \]
   This means that $m(N_{\min})$ is increasing, since $A_1 < A_2 < \dots < A_k$.
    
    We now show that $N_{\min}$ is the unique inclusion-maximal nested set whose NL-labeling is increasing and that $N_{\min}$ is the minimum nested set in the NL-order. Consider an inclusion-maximal nested set $N'= \{Y_1,Y_2,\ldots, Y_r\}$ with $N'\neq N_{\min}$ and whose NL-labeling is \[m(N') = (m_{N'}(Y_1), m_{N'}(Y_2), \ldots, m_{N'}(Y_r)).\]
    Since $N' \not= N_{\min}$, \Cref{lem:nl_unique} tells us there is some first place where $m(N)$ and $m(N_{\min})$ differ.
    By the proof of \Cref{lem:nl_unique}, this is also the first place where the sequences of flats differ, i.e., this is the first $k$ for which $X_k \not= Y_k$ in the ordered list of flats (note that we use the notation from the previous part of this claim, where $X_k \in N_{\min}$).
    
    We claim that $m_{N'}(Y_k)> m_{N_{\min}}(X_k)$. 
    For the sake of contradiction, suppose that $m_{N'}(Y_k)< m_{N_{\min}}(X_k)$. Our construction of $A_k$ implies that $m_{N'}(Y_k) \leq \bigvee_{i=1}^{k-1}X_i$. By \Cref{lem:incomp}, there exists an inclusion-maximal flat $Z$ of $\{X_1,\ldots, X_{k-1}\}$ containing $m_{N'}(Y_k)$.  It follows from the definition of $m_{N'}(-)$ that $Z\not< Y_k$. Also, as $m_{N'}(Y_k)\subseteq Y_k\cap Z$, \Cref{cor:disjoint} ensures that $Y_k$ and $Z$ are not incomparable. Thus we must have that $Y_k <Z$. We now show that this leads to a contradiction.

    If $Y_k <Z$, then $Y_k < \bigvee_{i=1}^{k-1} X_i$ and $\{Y_1, Y_2,\ldots, Y_k\}$ is a nested set of
    \[(\cL\vert_{\bigvee_{i=1}^{k-1} X_i},B\vert_{\bigvee_{i=1}^{k-1} X_i})\,.\] However, $\bigvee_{i=1}^{k-1} X_i = \bigvee_{i=1}^{k-1} A_i$ has rank $k-1$. By the first part of \Cref{lem:dim-nested-set}, this means that $\{Y_1, Y_2,\ldots, Y_k\}$ cannot be a nested set of $(\cL\vert_{\bigvee_{i=1}^{k-1} X_i},B\vert_{\bigvee_{i=1}^{k-1} X_i})$. We've thus reached a contradiction and shown that $m_{N'}(Y_k)> m_{N_{\min}}(X_k)$ and that $m(N_{\min})< m(N')$. 
    
    To see that $m(N')$ is not increasing, note that, by the third part of \Cref{lem:dim-nested-set}, there is a unique inclusion minimal flat $Y_j$ of $N'$ which contains $m_{N_{\min}}(X_k)$. We also know that $k< j$ and $m_{N'}(Y_j)\leq m_{N_{\min}}(X_k)< m_{N'}(Y_k)$. Combining these observations, we see $m(N')$ must have a descent at a flat $Y_i$ for some index $i$ with $k \leq i < j$. 
\end{proof}

\begin{example}
    Consider the building set $B=\{1,2,3,4,5,13,123,45\}$ of the Boolean lattice on $5$ elements. 
    The unique inclusion-maximal nested with increasing NL-labeling is $N_{\min} = \{1,2,123,4,45\}$. The NL-labeling of $N_{\min}$ is $m(N_{\min})=(1,2,3,4,5)$.
\end{example}

\subsection{Some Combinatorial Forestry}\label{subsec:forestry}
In this subsection, we describe how taking the link of a flat $Z\in \Delta(\cL,B)$ interacts with the NL-order, as constructed in \Cref{subsec:combo_order}. Given a nested set $N$ containing $Z$, let $\tau_Z(N)$ be the image of $N$ in the link of $Z$. We describe the forest poset structure of the image of $N$ in the link of $Z$ (Proposition \ref{lem:lumberjack}) and also the labeling function $m_{\tau_Z(N)}(-)$ of $\tau_Z(N)$ (\Cref{lem:link_labels}). These descriptions let us conclude that if $m(N)$ has a descent at a flat $X\neq Z$, then  $m(\tau_Z(N))$ has a descent
(\Cref{claim:min-element}). This will be a key fact in the inductive step of our proof of Theorem \ref{thm1}.

The following proposition describes the forest poset structure of $\tau_Z(N)$ in the link of $Z$. Roughly, it states that the forest of $\tau_Z(N)$ is constructed from $N$ by separating the tree that contains $Z$ into two trees. The first tree consists of all elements less than or equal to $Z$ and the second tree consists of (relabelings of) all of the elements in the original tree which are not less than or equal to $Z$. In the case of the maximal building set, this procedure separates a chain $C$ containing $Z$ into the two disjoint chains $C_{\leq Z} = \{X\in C: X\leq Z\}$ and $C_{>Z}= \{Y\setminus Z: Y\in C, Z< Y\}$.

\begin{lemma}\label{lem:partial-lumberjack}
If $Y$ covers $X$ in the forest poset of $N$, then $\tau_Z(Y)$ covers $\tau_Z(X)$ in the forest poset of $\tau_Z(N)$.
\end{lemma}

    \begin{proof}
    Suppose that $Y$ is a flat of $N$ covering $X$. It follows from the definition of $\tau_Z$ that $\tau_Z(X)< \tau_Z(Y)$. It now remains to show that there does not exist a flat $W\in N$ such that $\tau_Z(X) < \tau_Z(W) < \tau_Z(Y)$. To do so, we break into three separate cases based on the comparability of $Z$ to $X$.
  \begin{enumerate}
  \item {\sf Case 1}: If $X<Z$, then, as $N$ is a forest poset (Observation \ref{obs:forest}) and $Y$ covers $X$, we know that $Y\leq Z$. In this case, the forest poset of $\tau_Z(N)$ restricted to $\tau_Z(Y)$ is identical to the forest poset of $N$ restricted to $Y$. In particular, there does not exist a flat $W\in N$ such that $\tau_Z(X)< \tau_Z(W) <\tau_Z(Y)$.
  \item {\sf Case 2}: If $Z<X$, suppose that there did exist a flat $W\in N$ such that $\tau_Z(X) < \tau_Z(W) < \tau_Z(Y)$. In this case, $\tau_Z(X)= X\setminus Z$ and $\tau_Z(Y)=Y\setminus Z$. As $\tau_Z(X) \not\leq Z$, we must have that $\tau_Z(W) \not\leq Z$. Thus, $W\not\leq Z$ and $X < W\vee Z < Y$. If $Z\leq W$, then $W\vee Z= W$ and we obtain a contradiction to the fact that $Y$ covers $X$ in $N$. Thus we can conclude that $W$ and $Z$ are incomparable and $W\vee Z \not\in B$.
    Since $B$ is a building set and the fact that $\{W,Z\}$ are nested and incomparable, Lemma \ref{lem:incomp} tells us that the map
    \begin{align*}
        [\emptyset, W] \vee [\emptyset, X] & \to [\emptyset, W \vee X]\\
        (F,G) & \mapsto F \vee G
    \end{align*}
    must be an isomorphism.
    We claim, however, that this map cannot be an isomorphism. In particular, it is not injective. To see this, observe (a) that both $Z$ and $X$ lie in different intervals of the product on the left and (b) that $X\vee W = Z \vee W$. Both $(a)$ and $(b)$ follow directly from the facts that $Z< X$ and that $X< W\vee Z$.
  \item {\sf Case 3}:  If $Z$ and $X$ are incomparable, we proceed in a similar fashion to the previous case.

Suppose for contradiction that we had a flat $W \in N$ such that 
\[\tau_Z(X) < \tau_Z(W) < \tau_Z(Y)\,.\]
In order to derive a contradiction, we want to show that $X,W,$ and $Z$ are pairwise incomparable in $\cL$ and then derive a contradiciton using the definition of a nested set.
We already know that $X$ and $Z$ are incomparable. 
Now we just need to show that $W$ is incomparable to both $X$ and $Z$.
To that end, we first collect two useful observations:
\begin{itemize}
\item Corollary \ref{cor:disjoint} tells us that $X = \tau_Z(X)$, so the preceeding sequence of inequalities reads
\[X < \tau_Z(W) < \tau_Z(Y)\,.\]
This means that $\tau_Z(W)$ certainly cannot sit in $\cL|_Z$.
That is, either $Z<W$ or $Z$ and $W$ are incomparable.
\item Since $\tau_Z(X)< \tau_Z(W)$, this means that
\[
X\vee Z < W\vee Z\,,
\]
in the original lattice $\cL$.
\end{itemize}

\noindent Now we are ready to check that $W$ is incomparable to both $X$ and $Z$.

Let's start by comparing $W$ and $X$.
If $W$ sits below $X$, then the same holds when we take the wedge with $Z$, i.e.,
\[
W < X \qquad\Rightarrow\qquad W \vee Z < X \vee Z\,,
\]
contradicting our second observation.
On the other hand, if $W$ sits above $X$, then Observation \ref{obs:forest} implies that $W$ also sits above $Y$ and we have 
\[
Y < W \qquad\Rightarrow\qquad Y \vee Z \leq W \vee Z \qquad\Rightarrow\qquad \tau_Z(Y) \leq \tau_Z(W)\,,
\]
contradicting our assumption that $\tau_Z(W) < \tau_Z(Y)$.

Now let us compare $W$ and $Z$.
Our first observation tells us that $W$ does not sit below $Z$, so we just need to check that $W$ cannot sit above $Z$.
So suppose $Z \leq W$.
Then $W \vee Z = W$ and we have 
\[
X \vee Z < W \qquad \Rightarrow \qquad X < W\,,
\]
contradicting the fact that $X$ and $W$ are incomparable.

We've now shown that $\{X,W,Z\}$ are pairwise incomparable, so we are ready to derive our contradiction using the definition of building sets and the fact that $N$ is nested.

Since $B$ is a building set, \cite[Proposition 2.8.2]{FK} tells us that the map
 \begin{align*}
    \prod_{V \in \{W,X,Z\}} [\emptyset, V] & \to [\emptyset, W\vee X\vee Z]\\
    (F,G,H) & \mapsto F \vee G \vee H
\end{align*}
    must be an isomorphism.
    We claim that this map cannot be an isomorphism.
    However, we have 
    \[X\vee W \vee Z = W\vee Z = \emptyset \vee W \vee Z\,,\]
    since $X\vee Z < W\vee Z$.
    Thus this map is not injective and thus cannot be an isomorphism.
  \end{enumerate}
    \end{proof}

\begin{proposition} \label{lem:lumberjack}
Let $X,Y\in N\setminus \{Z\}$.
We have $X < Y$ in $\cL$ if and only if $\tau_Z(X) < \tau_Z(Y)$.
\end{proposition}

\begin{proof}
The forward direction follows directly from the definition of $\tau_Z$, so the only interesting content is the reverse direction.
Suppose for contradiction that $\tau_Z(X) < \tau_Z(Y)$, but $X$ is not less than $Y$.
We consider two cases, depending on whether or not $X \in \max(B)$.
\begin{enumerate}
    \item {\sf Case 1}. If $X\in \max(B)$, then Note \ref{note:maximalTau} tells us that $\tau_Z(X)$ is an inclusion maximal element of $B\vert_X \sqcup B^X$ and $\tau_Z(Y)$ was not larger than $\tau_Z(X)$ to begin with. 
    \item {\sf Case 2}. If $X\not\in \max(B)$, then there exists some flat $W\in N$ which covers $X$ in the forest poset of $N$. By Lemma \ref{lem:partial-lumberjack}, $\tau_Z(W)$ covers $\tau_Z(X)$ in the forest poset of $\tau_Z(N)$. But $\tau_Z(N)$ is a forest poset (Observation \ref{obs:forest}), so we cannot have that $\tau_Z(Y)$ also covers $\tau_Z(X)$.
\end{enumerate}
\end{proof}

The following lemma describes the labeling function $m_{\tau_Z(N)}(-)$ in terms of the labeling function $m_N(-)$. 

\begin{lemma}\label{lem:link_labels}
Let $X$ be a flat of $N$. The NL-labeling of $\tau_Z(X)$ in $\tau_Z(N)$ is 
    \[m_{\tau_Z(N)}(\tau (X))= \begin{cases}
    m_N(X), \text{ if $X\leq Z$}\\
    (m_N(X) \vee Z)\setminus Z, \text{ else.}
    \end{cases} \]
\end{lemma}

\noindent Note that the second case really is well-defined: $(m_N(X) \vee Z)\setminus Z$ is an atom of $\cL|_Z \times \cL^Z$, since it is an atom of $\cL^Z$.

\begin{proof}
We consider two cases, based on whether or not $X\leq Z$.
\begin{enumerate}
    \item {\sf Case 1}: Suppose that $X\leq Z$.
    Here the forest poset of $\tau_Z(N)$ restricted to $\tau_Z(X)$ is identical to the forest poset of $N$ restricted to $X$, so we are done.
    \item {\sf Case 2}: Suppose $X$ is not less than $Z$.
    The main difficulty of this case is to describe the atoms that sit below $\tau_Z(X)$.
    We handle this in the following sub-claim.

    \medskip

    \noindent {\bf Sub-Claim}. The set of atoms of $\cL\vert_Z\times \cL^Z$ contained in $\tau_Z(X)= (X\vee Z)\setminus Z$ is the set
  \[\{(A \vee Z) \setminus Z : A \leq X\}. \] 
  \begin{proof}[Proof of Sub-Claim]
      There are two cases to consider: either $Z < X$ or $X$ and $Z$ are incomparable.
      \begin{enumerate}
      \item {\sf Sub-case 1}: Let $Z<X$.
      Then $A \vee Z \leq X\vee Z = X$ if and only if $A \leq X$, so the statement is
      true.
      \item {\sf Sub-case 2}: Suppose $Z$ and $X$ are incomparable.
      From Lemma \ref{lem:incomp}, the map
      \begin{align*}
          [\emptyset, X]\times [\emptyset, Z] & \to [\emptyset, X\vee Z]\\
          (F,G) & \mapsto F\vee G
      \end{align*}
      is an isomorphism.
      In particular, if $A$ is an atom of $\cL$ with $A\not\leq Z$, then $A\vee Z \leq X\vee Z$ if and only if $A \leq X$.
  \end{enumerate}
  \end{proof}

  \noindent Armed with this sub-claim, we are now ready to prove {\sf Case 2}.
  By \Cref{lem:lumberjack}, a flat $\tau_Z(W)$ of $\tau_Z(N)$ is a child of  $\tau_Z(X)$ if and only if $W < X$ and $W\not\leq Z$.
  Together with our sub-claim, this implies that the set of atoms of $\cL\vert_Z \times \cL^Z$ in $\tau_Z(X)$ but not in any of its children is equal to
  \[\{(A\vee Z)\setminus Z: A \leq X, A\not\leq W \text{ for all } W< X\}. \]
Of these atoms, $(m_N(X)\vee Z)\setminus Z$ is the smallest.
\end{enumerate}
\end{proof}

The following claim tells us that descents in the NL-labeling are preserved by taking links of certain vertices.

\begin{claim}\label{claim:min-element}
Let $N$ be a nested set with whose NL-labeling has a descent at a flat $X\neq Z$. Then the NL-labeling of $\tau_Z(N)$ has a descent.
\end{claim}

\begin{proof}
Let $Y$ be the flat of $N$ such that $m_N(Y)$ immediately follows $m_N(X)$ in the NL-labeling of $N$. First, observe that $Y$ must cover $X$ in the forest poset of $N$. As $m_N(X)>m_N(Y)$, if at the time of picking $X$, we were able to pick $Y$, we would have. Thus $X$ must be a child of $Y$. Since $m_N(X)$ and $m_N(Y)$ are consecutive in the NL-labeling of $N$, there are no flats $W\in N$ such that $X<W<Y$.

By Proposition \ref{lem:lumberjack}, we have $\tau_Z(X) < \tau_Z(Y)$ and we have to pick $\tau_Z(X)$ before picking $\tau_Z(Y)$ when we construct $m(\tau_Z(N))$.
This means that $m_{\tau_Z(N)}(\tau_Z(X))$ precedes $m_{\tau_Z(N)}(\tau_Z(Y))$ in $m(\tau_Z(N))$.
By Lemma \ref{lem:link_labels} \[m_{\tau_Z(N)} (\tau_Z( X)) > m_{\tau_Z(N)} (\tau_Z(Y))\,.\] Thus there exists a flat $X'$ of $N$ such that $m(\tau_Z(N))$ has a descent at $\tau_Z(X')$ and $m_{\tau_Z(N)}(\tau_Z(X'))$ lies in between $m_{\tau_Z(N)}(\tau_Z(X)$ and $m_{\tau_Z(N)}(\tau_Z(Y))$.
\end{proof}

In general, we do not expect $m(\tau_Z(N)) = m(N)$, just that descents are preserved.
We illustrate some of the subtleties of this claim in the following example.

\begin{example}\label{localtauexample}
Let $\cL$ be the Boolean lattice on three elements and $B = \{1,2,3,13,12,123\}$.
Take the linear order $2<1<3$ on the ground set of the matroid.
Consider $N = \{3,13,123\}$ so that $m(N)= (m_N(3),m_N(13),m_N(123)) = (3,1,2)$. In the NL-labeling of $N$, there is a descent at $3$. \Cref{claim:min-element} tells us that there will be a descent at $m_{\tau_{13}(N)}(3)$ in $m(\tau_{13}(N))$. We can confirm this by computing 
   \[ m(\tau_{13}(N)) = (m_{\tau_{13}(N)}(2),m_{\tau_{13}(N)}(3),m_{\tau_{13}(N)}(13)) = (2,3,1)\]
Although $m(\tau_{13}(N)) \neq m(N)$, $m(\tau_{13}(N))$ still has a descent at $\tau_{13}(3)$.  

One subtlety of this section is that the relative NL-order of nested sets may not be preserved by taking links.
Take $N' = \{2,3,123\}$, for example.
We have $m(N') = (2,3,1)$ , so that $N$ comes before $N'$ in the NL-order (since $2<3$).
The image of these two nested sets in the image of the link of $Z = 3$ are
$\tau_Z(N) = \{1,12\}$ and $\tau_Z(N') = \{2,12\}$.
The corresponding $m$-vectors are $m(\tau_Z(N)) = (1,2)$ and $m(\tau_Z(N')) = (2,1)$.
Now $\tau_Z(N')$ comes before $\tau_Z(N)$ in the NL-order.
\end{example}

\subsection{Geometry}\label{subsec:geometry}

In this section, we will prove Theorem \ref{thm1} by induction.
The following lemma will be useful for the base case of this inductive proof.

\begin{lemma}\label{lem:dimension-calculation}
   Let $B$ be a building set of $\cL$ such that $|\max(B)|= \rank(M)-1$. The inclusion-maximal nested sets of $(\cL,B)$ are exactly the sets $\max(B) \cup \{X\}$ with $X\in B\setminus \max(B)$. Furthermore, if $X\in B\setminus \max(B)$, then $X$ is an atom of $\cL$.
\end{lemma}

\begin{proof}
     The first claim is immediate. For the second claim, suppose for the sake of contradiction that $X\in B\setminus \max(B)$ and that $X$ is not an atom. Then there is an atom $A$ strictly contained in $X$. By \Cref{prop:backman-danner}, $A\in B$ so $A\in B\setminus \max(B)$. But then $\max(B)\cup \{A,X\}$ is a nested set. This contradicts the first claim.
\end{proof}

We now prove that the normal complex order and the NL-order have the same minimum element in the case when $(\cL,B)$ has a one-dimensional normal complex.  This claim implies that the normal complex order and the NL-order are weakly locally equivalent.

\begin{claim}\label{claim:dimension-1-minimal}
Suppose that a normal complex $\cN= \cN_{\cL, B,\varphi}$ of $(\cL,B)$ is one-dimensional. The normal complex order and the NL-order have the same minimum element.
\end{claim}
\begin{proof}
    If $\cN$ is one-dimensional, then $|\max(B)| = \rank(M)-1$. Thus by \Cref{lem:dimension-calculation}, the set $B\setminus \max(B)$ consists entirely of atoms of $\cL$. Let $A_0$ be the smallest atom inside of $B\setminus \max(B)$. By \Cref{lem:dimension-calculation}, $N_{\min} = \max(B) \cup \{A_0\}$ is an inclusion-maximal nested set of $(\cL,B)$. We claim that $N_{\min}$ is the minimum nested set of the NL-order. By \Cref{prop:new_NL_min}, it suffices to check that $m(N_{\min})$ is increasing. This is true essentially by definition. The only thing we need to check is that, if $X$ is the flat of $\max(B)$ containing $A_0$, then $m_{N_{\min}}(A_0)=A_0< m_{N_{\min}}(X)$. This, in turn, follows directly from our choice of $A_0$.

    We now check that $N_{\min}$ is minimal in the normal complex order. Let $c\in \R^B$ be the vector defining $\cN$ as described in \Cref{lem:c-is-enough}. Let $N = \{A\} \cup \max(B)$ be an inclusion-maximal nested set of $(\cL,B)$. As the vertex $v_N$ lives in the relative interior of $\sigma_N$, we can write $v_N$ as 
    \[v_N = \lambda_A e_A + \sum_{X\in \max(B)} \gamma_{A,X} E_X \]
    where $\lambda_A$ is a positive real number and each $\gamma_{A,X}$ is a (possibly negative) real numbers such that 
    \begin{equation}\label{eq:equalities}
        \langle v_N, e_X \rangle = c_X \text{ for all } X\in N.
    \end{equation} 
    We include the subscripts to indicate that $\lambda_A$ depends on $A$ and that $\gamma_{A,X}$ depends on $A$ and $X$.

    Let $i_0$ be the smallest element of $E$ contained in $A_0$. We claim that $(v_{N_{\min}})_{i_0} > (v_{N})_{i_0}$ and $(v_{N_{\min}})_{j} = (v_{N})_{j}$ for all inclusion-maximal nested sets $N\neq N_{\min}$ and indices $j<i_0$. If we prove this, it will imply that $N_{\min}$ is the minimum nested set in the normal complex order. 

     We now unwind the equalities of \Cref{eq:equalities}. Again, let $N= \{A\} \cup \max(B)$ be a nested set. Let $X$ be the unique flat of $\max(B)$ containing $A$, which is guaranteed to exist by \Cref{lem:dim-nested-set}. There are three cases to analyze:
     \begin{enumerate}
     \item If $i$ is an index not contained in $X$, then $(v_N)_i = c_Y/|Y|$ where $Y$ is the flat of $\max(B)$ containing $i$. 
     \item If $i\in X\setminus A$, then $(v_N)_i < c_X/|X|$.
     \item If $i\in A$, then $(v_N)_i > c_X/|X|$
     \end{enumerate}
     It is straightforward, although laborious, to check that these three cases imply our claim. We leave the details of this check to the reader.
\end{proof}

\begin{claim}\label{claim:min-match}
    The normal complex order and the $NL$-order have the same minimum nested set.
\end{claim}

\begin{proof}
Let $\cN$ be a normal complex of $(\cL,B)$. We do induction on the dimension of $\cN$.
There are two base cases to consider: dimension $0$ (trivial, since it is a point) and dimension $1$, which follows from \Cref{claim:dimension-1-minimal}.

For the general case, let $N_{\min}$ be the minimum nested set in the NL-order and $N_\gamma$ be the minimum nested set in the normal complex order.
Suppose for contradiction that $N_{\min} \not= N_{\gamma}$ and hence that $N_\gamma$ is not minimal in the NL-ordering. By \Cref{prop:new_NL_min}, there are some flats $X,Y\in {N_\gamma}$ with $m_{N_\gamma}(X)>m_{N_\gamma}(Y)$ and $m({N_\gamma})= (\ldots, m_{N_\gamma}(X), m_{N_\gamma}(Y),\ldots )$. Choose any flat $Z\in {N_\gamma}$ which is not in $\max(B)$ and not equal to $X$.
Such a flat exists because the normal complex is at least two-dimensional. 
Let $\tau_Z$ be the bijection from \Cref{thrm:wondertopes}. By \Cref{claim:min-element}, the NL-labeling of $\tau_Z(N_\gamma)$ contains a descent. Thus by \Cref{prop:new_NL_min}, $\tau_Z(N_\gamma)$ is not minimal in the NL-order of $(\cL\vert_Z \times \cL^Z, B\vert_Z \sqcup B^Z)$. By the first part of \Cref{lem:intersection}, the normal complex $\cN_Z$ of $(\cL\vert_Z \times \cL^Z, B\vert_Z \sqcup B^Z)$ is smaller-dimensional than $\cN$.  Thus our induction hypothesis tells us that $\tau_Z(N_\gamma)$ is not minimal in the normal complex order of $(\cL\vert_Z \times \cL^Z, B\vert_Z \sqcup B^Z)$ and there exists a nested set $N'$ of $(\cL,B)$ such that $\tau_Z(N')$ precedes $\tau_Z(N_\gamma)$ in the normal complex order of $(\cL\vert_Z \times \cL^Z, B\vert_Z \sqcup B^Z)$. By \Cref{prop:lex_link}, this contradicts our assumption that $N_\gamma$ is the minimum nested set in the normal complex order of $(\cL,B)$.
\end{proof}

\begin{proof}[Proof of Theorem \ref{thm1}]
We proceed by induction on the dimension of the normal complex.
There are two base cases: dimensions $0$ (which is trivial, since the complex is a point) and $1$.
When the normal complex has dimension $1$, \Cref{lem:dimension-calculation} tells us that the nested set complex looks like a star with center $\max(B)$.
In particular \emph{any} order is a shelling order and thus the normal complex order is a shelling order of the nested set complex.

For the inductive step, assume that $N' <_{\gamma} N$. We'll show that there exists $N'' <_{\gamma} N$ such that
\[
N \cap N' \subseteq N \cap N''  \quad \text{and}\quad \vert N\cap N''\vert = \vert N \vert -1\,.
\]
We consider two cases, based on whether or not $N\cap N'= \max(B)$.
\begin{enumerate}
    \item {\sf Case 1}: $N\cap N'$ equals $\max(B)$.
Since $N$ is not minimal in the normal complex order, \Cref{claim:min-match} tells us that it is not minimal in the NL-order, i.e., $N \not= N_{min}$. By \Cref{prop:new_NL_min}, there are some flats $X,Y\in N$ with $m_N(X)>m_N(Y)$ and $m(N)= (\ldots, m_N(X), m_N(Y),\ldots )$. Choose any flat $Z\in N$ which is not in $\max(B)$ and not equal to $X$. Such a flat exists because the normal complex $\cN$ is at least two dimensional. Let $\tau_Z$ be the bijection from \Cref{thrm:wondertopes}. By \Cref{claim:min-element}, the NL-labeling of $\tau_Z(N)$ contains a descent. Thus by \Cref{prop:new_NL_min}, $\tau_Z(N)$ is not minimal in the NL-order of $(\cL\vert_Z \times \cL^Z, B\vert_Z \sqcup B^Z)$. By \Cref{claim:min-match}, $\tau_Z(N)$ is also not minimal in the normal complex order of $\Delta(\cL\vert_Z \times \cL^Z, B\vert_Z \sqcup B^Z)$. By the first part of \Cref{lem:intersection}, the normal complex of $(\cL\vert_Z \times \cL^Z, B\vert_Z \sqcup B^Z)$ is smaller-dimensional than $\cN$. Thus our inductive hypothesis tells us that there exists a nested set $N''$ of $(\cL,B)$ such that: 
\begin{itemize}
\item $\tau_Z(N'')$ is smaller than $\tau_Z(N)$ in the normal complex order of $(\cL\vert_Z \times \cL^Z, B\vert_Z \sqcup B^Z)$ and
\item $\vert \tau_Z(N'')\cap \tau_Z(N)\vert =\vert \tau_Z(N)\vert -1 $.
\end{itemize}
\noindent By \Cref{prop:lex_link}, the first bullet point lets us conclude that $N''$ is smaller than $N$ in the normal complex order.  As $\tau_Z$ induces an isomorphism of nested set complexes, the second bullet implies that $\vert N'' \cap N\vert  =  \vert N \vert -1$.

\item {\sf Case 2}: $(N\cap N')\setminus \max(B)\not=\emptyset$.
Take $Z \in (N\cap N')\setminus \max(B)$.
From Lemma \ref{lem:intersection}, the intersection of $\cN$ with $H_{X,\varphi}$ (i.e., the facet of $\cN$ defined by $H_{X,\varphi}$) is a smaller normal complex.
Namely, a normal complex of $(\cL\vert_Z \times \cL^Z, B\vert_Z \sqcup B^Z)$. Let $\tau_Z$ be the bijection from \Cref{thrm:wondertopes}. By \Cref{prop:lex_link}, $\tau_Z(N')$ precedes $\tau_Z(N)$ in the normal complex order of $(\cL\vert_Z \times \cL^Z, B\vert_Z \sqcup B^Z)$. By the first part of \Cref{lem:intersection}, the normal complex of $(\cL\vert_Z \times \cL^Z, B\vert_Z \sqcup B^Z)$ is smaller-dimensional than $\cN$. Thus our induction hypothesis tells us that there exists some nested set $N''$ such that $\tau_Z(N')<_{\cN} \tau_Z(N'') <_{\cN} \tau_Z(N)$, $\tau_Z(N')\cap \tau_Z(N) \subseteq \tau_Z(N'') \cap \tau_Z(N)$ and $\vert \tau_Z(N'')\cap \tau_Z(N)\vert =\vert \tau_Z(N)\vert -1 $. For identical reasons to the first case, this implies our claim.
\end{enumerate}
\end{proof}

\bibliography{references}{} 
\bibliographystyle{plain}

\end{document}